\documentclass[11pt,reqno]{amsart}

\usepackage{color}

\usepackage{amssymb,latexsym}
\usepackage{graphicx}
\usepackage{amsmath,amsthm}
\usepackage{enumerate}
\usepackage{fancyhdr,amssymb}
\usepackage{url}	
\usepackage{tikz}
\usepackage{pgf}
\usepackage{amssymb,latexsym}
\usepackage{graphicx}
\usepackage{amsmath,amsthm}
\usepackage{enumerate}
\usepackage{fancyhdr,amssymb}
\usepackage{url}	
\usepackage{tikz}
\usepackage{pgf}

\usepackage{tipa}
\usepackage{amsmath,amsfonts}
\usepackage{mathtools}
\usepackage{booktabs}

\newcommand{\Z}{{\mathbb Z}}
\newcommand{\N}{{\mathbb N}}

\newcommand{\Q}{{\mathbb Q}}

\newcommand{\CC}{{\mathbb C}}

\newcommand{\Gg}{{\mathcal G}}

\newcommand{\Nc}{{\mathcal N}}

\newcommand{\Oo}{{\mathcal O}}

\newcommand{\Cc}{{\mathcal C}}

\newcommand{\Pc}{{\mathcal P}}
\newcommand{\Rc}{{\mathcal R}}
\newcommand{\Sc}{{\mathcal S}}
\newcommand{\Tc}{{\mathcal T}}

\newcommand{\X}{{\mathcal X}}
\newcommand{\Y}{{\mathcal Y}}
\newcommand{\Zc}{{\mathcal Z}}

\newcommand{\ta}{{\theta}}

\newtheorem{thm}{Theorem}[section]

\newtheorem{lema}[thm]{Lemma}

\newtheorem{ques}[thm]{Question}

\newtheorem{deff}[thm]{Definition}
\newtheorem{rem}[thm]{Remark}

\newtheorem{exam}[thm]{Example}

\numberwithin{equation}{section}

\begin{document}

\setcounter{page}{1}

\title{rational $\ta$-parallelogram envelopes via $\ta$-congruent elliptic curves}
\author{Sajad Salami \and Arman Shamsi Zargar}
\address{Inst\'{i}tuto da Matem\'{a}tica e Estat\'{i}stica, Universidade Estadual do Rio de Janeiro (UERJ), Rio de Janeiro, Brazil}
\email{sajad.salami@ime.uerj.br}
\address{Department of Mathematics and Applications, University of Mohaghegh Ardabili, Ardabil, Iran}
\email{zargar@uma.ac.ir}

\maketitle

\begin{abstract}
We introduce a new generalization of $\ta$-congruent numbers by defining the notion of rational $\ta$-parallelogram envelope for a positive integer $n$, where 
$\ta\in (0, \pi)$ is an angle with rational cosine. Then, 
we study more closely some  problems related to the rational $\ta$-parallelogram envelopes,  
using the arithmetic of algebraic curves. 
Our results generalize the recent work of T.~Ochiai, where only the case $\ta=\pi/2$ was considered. 
Moreover, we answer the open questions  in his paper and their generalizations
for  any Pythagorean angle.
\end{abstract}
\bigskip

{\bf Subjclass 2020: }{Primary 11G05; Secondary  14H52} 

{\bf keywords}: Rational $\ta$-parallelogram envelope, $\ta$-congruent number, elliptic curve.

\section{Introduction}
\label{intro}
A positive integer $n$ is called a {\it congruent number}  if
it is  equal to the area of a right triangle with rational sides.
Equivalently, if there exist positive rational numbers $a$, $b$, and
$c$ such that $a< b <c,$ and
\begin{equation}
\label{eq1}
a^2+b^2=c^2, \quad  ab=2n.
\end{equation}
Determining all the congruent numbers is an old problem in Number Theory.
There are various sorts of generalizations of the problem, 
see for example \cite{fujw1,fujw2,kob1,topyui,ShYsh1,ShYsh2}.

In \cite{fujw1}, M.~Fujiwara introduced and studied an interesting generalization of congruent numbers, called $\ta$-congruent numbers. Afterwards, several authors (\cite{djps,fujw2,j-s,kan1,Moranki,Ochi}) studied this new concept with different approaches using the arithmetic of elliptic curves \cite{kob1,slm1}.
In order to describe this generalization, let  $\ta\in(0, \pi)$ be an angle with rational cosine, i.e., it satisfies  $\cos(\ta)=s/r$ with $r, s \in \Z$  such that $0\leq |s| <r$ and $\gcd (r,s)=1$. 
Let $\N $ be the set of natural numbers. An element $n$ of $\N$ is called a {\it $\ta$-congruent number} if there exists a triangle with rational sides and area equal to $n \sqrt{r^2-s^2}$. Equivalently, if there are three positive rational numbers $a$, $b$, and $c$ satisfying $a\leq  b <c$, and
\begin{equation*}
\label{eq2}
a^2+b^2-\frac{2s}{r}ab=c^2, \quad  ab=2rn.
\end{equation*}
We denote such a triangle by  $(a, b, c)$ and call it a {\it rational  $\ta$-triangle for $n$}.
It is clear that if a positive integer $n$ is $\ta$-congruent with a  $\ta$-triangle $(a, b, c)$, then
$n m^2$ is also a  $\ta$-congruent number with the $\ta$-triangle $(ma, mb, mc)$.
Hence, one may concentrate on the square-free positive integers.
The problem of determining $\ta$-congruent numbers is
related to finding non-2-torsion points on the elliptic curve given by the following Weierstrass  equation,
\begin{equation*}
\label{eqt}
E_\ta^n: y^2=x(x+(r+s)n)(x-(r-s)n),
\end{equation*}
where $r$ and $s$ are as above. One can see  \cite{fujw1,fujw2,kan1} for more details.

In a recent work, T.~Ochiai \cite{Ochi} generalized the congruent numbers in a novel way.
Indeed, he considered  the set $\Nc$  of all positive integers $n$ such that there is an {\it envelope for $n$}, i.e., a quintuple $(a,b,c,d,e)$ of positive rational numbers satisfying
\begin{equation}
\label{eq3}
a^2+b^2=c^2, \quad  a^2+d^2=e^2, \quad a(b+d)=n.
\end{equation}
We note that any congruent number belongs to the set $\Nc$, because if we let $b=d$, then 
Equations~\eqref{eq3} are equivalent to \eqref{eq1} for $4n$  and therefore for $n$. It is proved that $n\in\Nc$ if either $n$ or $2n$ is a congruent number. Moreover, it is shown that
$n\in\Nc$ if and only if  a certain set of algebraic equations
has some rational solutions, see Theorem~2 in \cite{Ochi}.
Given any positive rational number $m$, T.~Ochiai also considered the set  $\Nc(m)$  consisting of all positive integers $n \in \Nc$ such that there exists an envelope $(a, b, c, d, e)$ for $n$ with $d=mb$.
By studying the set of rational points on certain family of elliptic curves and a genus~five algebraic curve,
he provided certain conditions which lead to deciding whether a given positive integer $n$ belongs to $\Nc(m)$.
In the end of his paper \cite{Ochi}, T.~Ochiai asked the following questions concerning  the sets $\Nc$ and $\Nc(m)$.
\begin{ques}
	\label{ques1}
	Notation being as above,  one may ask that:
	\begin{itemize}
		\item [(i)]  Is $\Nc=\N$?
		\item [(ii)] Given any  $n\in \Nc$, are there infinitely many distinct envelopes for $n$?
		\item [(iii)] Given any $n\in \Nc$, are there infinitely many positive rational numbers $m$ such that $n\in\Nc(m)$?
	\end{itemize}
\end{ques}

In this paper, by generalizing the notion of envelope in T.~Ochai's paper \cite{Ochi}, 
we introduce naturally a real-life geometric object, namely, the  {rational \it $\ta$-parallelogram envelope for $n$}.
This new notion can be viewed as a generalization of $\ta$-congruent numbers which involves the rational $\ta$-triangles. 
We will study the set of $\ta$-parallelogram envelopes for positive numbers by investigating  the set of rational points on a certainly defined algebraic curve.
Moreover, we shall respond affirmatively to Question~\ref{ques1} in a more general setting related to the
set of $\ta$-parallelogram envelopes which arises through this study.

\section{The rational $\ta$-parallelogram envelopes}
Let  $\ta\in(0, \pi)$ be an angle with rational cosine, i.e.,    $\cos(\ta)=s/r$
with $r, s \in \Z$ such that $0\leq |s| <r$ and $\gcd (r, s)=1$.
We call  $\ta$ a {\it Pythagorean angle} if its sine is also  rational,  say
$\sin(\ta)= t/r$ for which  $0\leq |t|=\sqrt{r^2-s^2 } <r$ and $\gcd (r,t)=1$.
Then, we start by the following definition.
\begin{deff}	
	\label{def:2.1}
	Let  $\ta\in(0, \pi)$ be  angle with rational cosine. 
	We define   $\Nc_\ta$ to be the set  of all $n \in \N$ such that there exist  positive rational numbers $a, b, c, d, $ and $e$ satisfying
	\begin{equation}
	\label{eq4}
	a^2+b^2 -\frac{2s}{r} ab =c^2, \quad
	a^2+d^2 +\frac{2s}{r} ad =e^2, \quad
	a(b+d)  = rn.
	\end{equation}
\end{deff}

For any $n\in\Nc_\ta$, we denote such a quintuple of positive rational numbers satisfying  Equations~\eqref{eq4} 
by  $(a, b,  c, d, e)_\ta$  and call it a {\it rational $\ta$-parallelogram envelope for $n$}.
This notion with $\ta=\pi/2$ can be viewed as a generalization of the envelopes
defined by T.~Ochiai in  \cite{Ochi}.  Figure~1 shows a $\ta$-parallelogram envelope for $n$.
\begin{figure}[h]
	\includegraphics[height=.09\textheight, width=.40\textwidth]{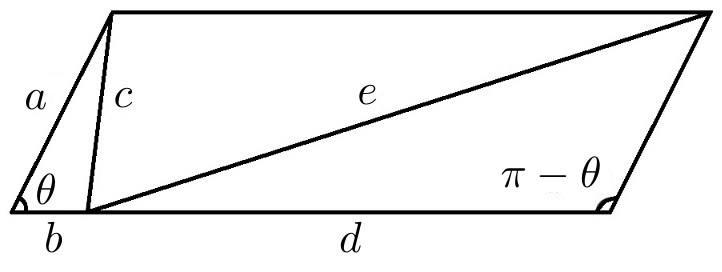}
	\caption{A rational $\ta$-parallelogram envelope for $n$}
\end{figure}

\begin{rem}
	\label{rem1}
	The following facts can be easily deduced from Definition~\ref{def:2.1}.
	\begin{itemize}	
		\item [(1)] A natural number $n$ belongs to $\Nc_\ta$ if and only if $n m^2 \in \Nc_\ta$ for any integer $m$.
		Hence, the area of a rational $\ta$-parallelogram envelope for $n$ is equal to $n \sqrt{r^2-s^2} \bmod (\Q^*)^2$.
		\item [(2)] A quintuple $(a, b,  c, d, e)_\ta$ is a rational $\ta$-parallelogram envelope for $n$ if and only if the quintuple $(a,  d, e, b,  c)_\ta$ is a rational $(\pi-\ta)$-parallelogram envelope for $n$.	
		\item [(3)]  If $n$  is  both  $\ta$- and
		$(\pi-\ta)$-congruent number, then $n\in\Nc_\ta$, but the converse  is not true.
		For instant, $n=7$ is neither a $\pi/3$-congruent number  nor a
		$2\pi/3$-congruent number, but there is a  rational $\pi/3$-parallelogram envelope for $n=7$ given by
		$$(a, b, c, d, e)_{\frac{\pi}{3}}=\left(\frac{143}{42}, \frac{20}{7}, \frac{19}{6}, \frac{1256}{1001}, \frac{3583}{858}\right)_{\frac{\pi}{3}}.$$
	\end{itemize}
\end{rem}

By the above remarks,  we  restrict our study to the case $\ta \in (0, \pi/ 2]$ and the square-free natural numbers $n$ throughout this work. Furthermore, by a   $\ta$-parallelogram envelope for an $n$ we always mean the rational one.

In the next definition, we consider  a subset of $\Nc_{\ta}$ which involves  the $\ta$-parallelogram envelope $(a,b,c,d,e)_{\ta}$
having an   angle  $\tau$ with rational cosine between the sides $c$ and $e$.
\begin{deff}
	Let $\tau \in(0,\pi)$ be an angle with rational cosine. We denote by  $\Nc_\ta(\tau)$  the set
	of all  $n\in \Nc_\ta$ such that there exist positive rational numbers  $a, b, c, d,$ and $e$ satisfying
	\eqref{eq4} and
	$$\cos\tau=\frac{a^2-bd - s a(b-d)/r}{ce}.$$
	For any $n \in \Nc_\ta(\tau)$, we denote  such a  quintuple by  $(a, b, c, d, e)_\ta^\tau$  and
	call it a $\ta$-parallelogram envelope with an angle $\tau$ for $n$.
\end{deff}
Figure~2 shows a $\ta$-parallelogram envelope with an angle $\tau$ for an element $n $ of $\Nc_\ta(\tau)$.
\begin{figure}[h]
	\begin{center}
	\includegraphics[height=.1\textheight,width=.42\textwidth]{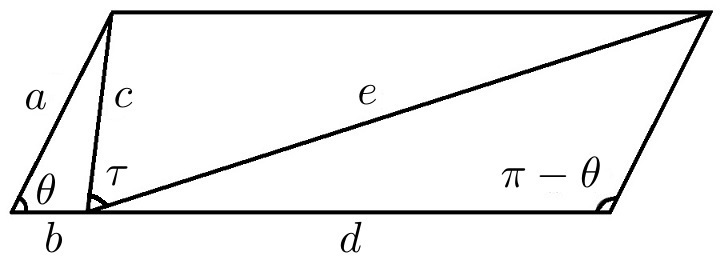}
		\caption{A $\ta$-parallelogram envelope with angle $\tau$ for $n$}
	\end{center}
\end{figure}

The following theorem, proved in Section~\ref{Pr0},  gives a relation between the set  $\Nc_\ta(\tau)$ and the $\tau$-congruent numbers.
\begin{thm}
	\label{main0}
	Let  $\ta \in (0,\pi/2]$ and  $\tau \in (0, \pi)$ be angles with rational cosines satisfying
	$\sin(\ta)/\sin(\tau)\equiv n' \ \bmod \ (\Q^*)^2$ for some   positive integer $n'$.
	Then, $n\in \N$ belongs to $\Nc_\ta(\tau)$ if and only if $2n n'$ is a $\tau$-congruent number. In particular, $n\in\Nc_{\ta}(\ta)$ if and only if $2n$ is a $\ta$-congruent number. Equivalently,
	$n$ is a $\ta$-congruent number if and only if $2n\in\Nc_{\ta}(\ta)$.
\end{thm}

\section{The $\ta$-congruent number elliptic curve and the set $\Nc_{\ta}$}
In this section, we provide an algebraically  necessary and sufficient condition for $n\in \N$ to be an element of $\Nc_\ta$ using the $\ta$-congruent number elliptic curves.
\begin{thm}
	\label{main1}
	A natural number $n$ belongs to $\Nc_\ta$ if and only if the  equations,
	\begin{equation}
	\label{eq5}
	\left\{
	\begin{array}{l}
	\hspace{-.1cm} E_\ta^w:  y^2=x(x+(r+s)w)(x-(r-s)w), \vspace{.2cm}\\
	\hspace{-.1cm} F_\ta^N:  v^2=u(u-(r+s)N)(u+(r-s)N), \vspace{.2cm}\\
	\hspace{-.1cm} C_R:   xv=uy,
	\end{array}
	\right.
	\end{equation}
	have a rational solution $(u,v,w,x,y)$ where $N=2n-w$,  $yv \not =0$ and $0 < w \leq n$.
\end{thm}

Taking $w=n$ in Theorem~\ref{main1} leads to the fact that if a natural number $n$ is simultaneously  $\ta$- and ($\pi-\ta$)-congruent number,
then  $n\in\Nc_\ta$.
The existence of a rational  solution $(u,v,w,x,y)$ in Theorem~\ref{main1} means that there exists a natural number $0 < w \leq n$
such that both the elliptic curves $E_\ta^w$ and $F_\ta^N$ have some non-$2$-torsion points $(x,y)$ and $(u,v)$  satisfying $x/y=u/v$, where the letter ``R''  in $C_R$ refers to the word ``Ratio".
We note that the ratio $b/d$ of a  $\ta$-parallelogram envelope $(a, b, c, d, e)_\ta$ for $n$ depends on the parameter $w$.
For $0 < w \leq n$, the point $P_w$ moves from the point $A_1$
to the middle point    $P_n$ of the side $A_1 A_2$, see Figure~3 below.
\begin{figure}[h]
		\begin{center}
		\includegraphics[height=.40\textheight,width=.50\textwidth]{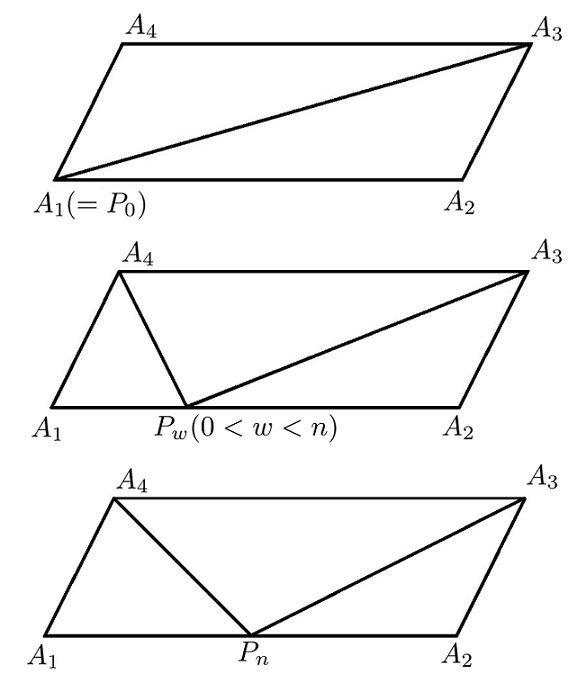}
		\caption{The role of $w$ in $\ta$-parallelogram envelope for $n$}
	\end{center}
\end{figure}

In order to reduce the number of variables in \eqref{eq5}, we let $v=uy/x$ in $F_\ta^N$  to obtain the following equivalence for Theorem~\ref{main1}.
\begin{thm}
	\label{main2}
	A natural number $n$ belongs to $\Nc_\ta$ if and only if the following equations,
	\begin{equation}
	\label{eq6}
	\left\{
	\begin{array}{l}
	\hspace{-.1cm} E_\ta^w:   y^2=x(x+(r+s)w)(x-(r-s)w), \vspace{.2cm}\\
	\hspace{-.1cm} G_\ta^N:  z^2= \left(x^2 + 2s(N+w)x - (r^2-s^2) w^2\right)^2+4(r^2-s^2)N^2 x^2,
	\end{array}
	\right.
	\end{equation}
	have a rational solution $(w,x,y,z)$ where $N=2n-w$, $y \not =0$, and $0 < w \leq n$.
\end{thm}

The proofs of Theorems~\ref{main1} and \ref{main2} are included in Section~\ref{Pr1-2}

\section{The $\ta$-parallelogram envelopes with ratio $m$}
In this section, we are going to study $\Nc_{\ta, m}$ a  subset of  $\Nc_\ta$ associated to a given
positive rational number $m$.
\begin{deff}
	For a positive rational number $m$, let  $\Nc_{\ta, m}$ denote the subset of $\Nc_\ta$ consisting of
	all natural numbers $n$ such that there exist  positive rational numbers $a$, $b$, $c$, $d$, $e$
	satisfying~\eqref{eq4} and  $d=mb$.
	We call such a quintuple a {\it $\ta$-parallelogram envelope with ratio $m$ for $n$} and denote it by
	$(a,b,c,d,e)_{\ta}^m$, see Figure~4.
\end{deff}
\begin{figure}[h]
	\begin{center}
		\includegraphics[height=.1\textheight,width=.42\textwidth]{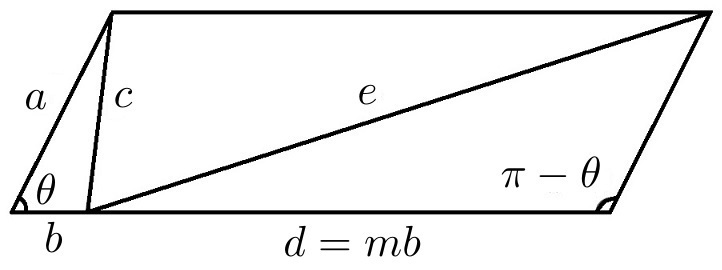}
		\caption{A $\ta$-parallelogram envelope with ratio $m$ for $n$}
	\end{center}
\end{figure}

For a positive integer $m$,
a natural number $n$ belongs to  $\Nc_{\ta, m}$  if and only if $nm$ is an element of  $\Nc_{\pi - \ta, m}$.
Indeed, if $(a, b, c, d, e)_{\ta}^m$ is a $\ta$-parallelogram envelope with ratio $m$ for $n$, then  we obtain a
$(\pi- \ta)$-parallelogram envelope with ratio $m$ as
$(d , a, e, ma, mc)_{\pi- \ta}^{m}$ for $mn$, by multiplying the first equation of \eqref{eq4} by $m^2$ and the third by $m$, see Figure~5.
\begin{figure}[h]
		\begin{center}
		\includegraphics[height=.37\textheight,width=.49\textwidth]{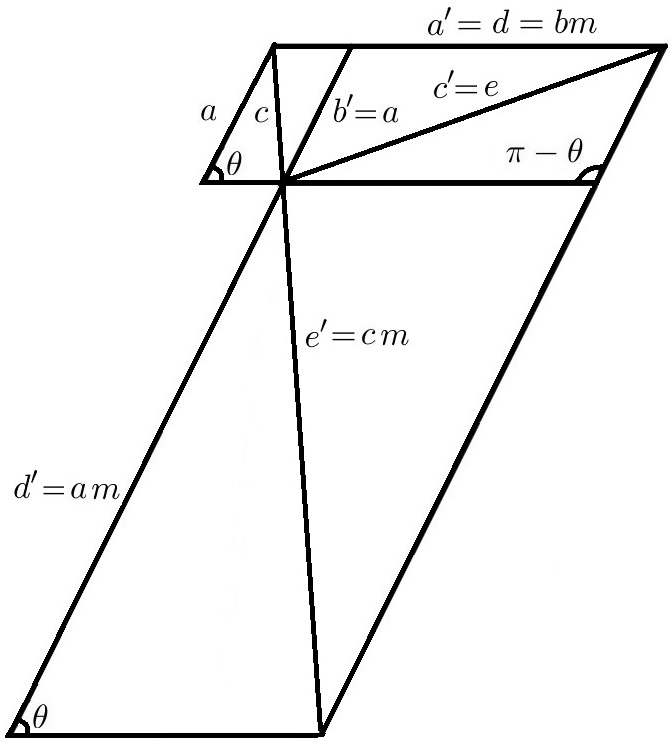}
		\caption{Two $\ta$-parallelogram envelopes with ratio $m$ for $n$  and $m n$}
	\end{center}
\end{figure}

The following theorem is a consequence of Theorems~\ref{main1} and \ref{main2}, which is proved in Section~\ref{Pr1-2}.
\begin{thm}
	\label{main3}
	A natural number $n$ belongs to $\Nc_{\ta, m}$ for some positive rational number $m$ if and only if
	the simultaneous equations,
	\begin{equation}
	\label{eq6a}
	\left\{
	\begin{array}{l}
	\hspace{-.1cm}  E_\ta^{(m,n)}:   \displaystyle y^2=x\left(x+\frac{2n(r+s)}{m+1}\right)\left(x-\frac{2n(r-s)}{m+1}\right),\vspace{.2cm}\\
	\hspace{-.1cm} G_\ta^{(m,n)}:  z^2= x^4+ b_1 x^3+ b_2 x^2+ b_3 x + b_4,
	\end{array}
	\right.
	\end{equation}
	have a rational solution $(x,y,z)$ with $y \not =0$, where
	$$
	\begin{array}{ll}
	b_1= 8 n s, & b_2 =\displaystyle\frac{8 n^2 (2 m^2r^2+4s^2 m + 3 s^2-r^2)}{(m+1)^2}, \vspace{.2cm} \\
	b_3  = -\displaystyle\frac{32 n^3 s (r^2-s^2) }{(m+1)^2}, & b_4   = \displaystyle\frac{16 n^4(r^2-s^2)^2}{(m+1)^4}.
	\end{array}
	$$
\end{thm}

\section{An elliptic curve related to $\Nc_{\ta, m}$}
\label{ellGtm}
In order to investigate  the set of rational points satisfying Equations~\eqref{eq6a},  we prove
the following  two results on the quartic  curve $G_\ta^{(m,n)}$ in
Section~\ref{main4-5}.
\begin{thm}
	\label{main4}
	The quartic curve  $G_\ta^{(m,n)}$ in the statement of  Theorem~\ref{main3} can be birationally transformed into the elliptic curve given by the following cubic equation,
	\begin{equation*}
	\label{eq:gtam}
	{\mathcal G}_\ta^m: \ {\mathcal Y}^2={\mathcal X}^3+(r^2m^2+2s^2m+r^2){\mathcal X}^2+m^2(r^2-s^2)^2{\mathcal X}.
	\end{equation*}
	Moreover, for  $\ta \in (0,\pi/2]$ with rational cosine, the  rank  $r_\ta(m)$ of the Mordell-Weil group of rational points on
	${\mathcal G}_\ta^m$ is at least one for all but finitely many $m \in \Q$ with an independent point given by
	$$\Pc=(-(r^2-s^2) m^2, s (r^2-s^2) m^2  (m+1)).$$
\end{thm}

Given  $\ta \in (0,\pi/2]$ with $\cos(\ta)=s/r \in \Q^*$, we
define the following quantities:
\begin{equation}
\label{Ms}
\begin{array}{l}
M_0=m(r^2-s^2),\vspace{.2cm}\\
M_1=r(m+r)\left(r(m+1)+2\sqrt{M_0}\right),\vspace{.2cm}\\
M_2=r(m+r)\left(r(m+1)-2\sqrt{M_0}\right).
\end{array}
\end{equation}
We note that $M_0$, $M_1$, and $M_2$  are all positive numbers.
Evidently $M_0,$ and $M_1>0$, since $0\leq |s|<r$ and  $m>0$. To show $M_2>0$, it is enough to observe
$$rm+r-2\sqrt{M_0}>0\Longleftrightarrow r^2(m+1)^2>4m(r^2-s^2)\Longleftrightarrow r^2(m-1)^2+4ms^2>0$$
which is the case.

In the next  result, we determine  all possibilities for
${\mathcal G_\ta^m}(\Bbb Q)_{\text{tors}}$, i.e., the torsion subgroup of  $\mathcal G_\ta^m$.
\begin{thm}
	\label{main5}
	Keeping the above notations, we have:
	$${\mathcal G_\ta^m}(\Bbb Q)_{\rm{tors}}\cong
	\left\{\begin{array}{ll}
	\hspace{-.1cm}\displaystyle\frac{\Bbb Z}{8\Bbb Z} & \text{if~}\sqrt{M_0}~\text{and either}~\sqrt{M_1}~\text{or}~\sqrt{M_2} \in \Q^*
	\vspace{.2cm}\\
	\hspace{-.1cm}	\displaystyle\frac{\Bbb Z}{2\Bbb Z}\times\frac{\Bbb Z}{8\Bbb Z} & \text{if~}\sqrt{M_0}, \sqrt{M_1},\text{and}~\sqrt{M_2} \in \Q^*
	\vspace{.2cm}\\
	\hspace{-.1cm}	\displaystyle\frac{\Bbb Z}{4\Bbb Z}~\text{or}~\frac{\Bbb Z}{2\Bbb Z}\times\frac{\Bbb Z}{4\Bbb Z} & \text{otherwise}.
	\end{array}\right.$$
	
	The points $(M_0, \pm r(m+1)M_0)$ are of order~$4$ and   points with the following
$\X$-coordinates
	\begin{equation*}
	\label{octic}
	\begin{array}{l}
	{\mathcal X}_1=M_0+\left(r(m+1)+\sqrt{M_1}\right)\sqrt{M_0},\vspace{.2cm} \\
	{\mathcal X}_2=M_0+\left(r(m+1) -\sqrt{M_1}\right)\sqrt{M_0},\vspace{.2cm}\\
	{\mathcal X}_3=M_0-\left(r(m+1)+\sqrt{M_2}\right)\sqrt{M_0}, \vspace{.2cm}\\
	{\mathcal X}_4=M_0-\left(r(m+1)-\sqrt{M_2}\right)\sqrt{M_0},
	\end{array}
	\end{equation*}
	are of order~$8$ in the torsion subgroup of ${\mathcal G_\ta^m}$.
	
	Moreover, for  $\ta\neq \pi/2$, the cases $\Bbb Z/4\Bbb Z$, $\Bbb Z/8\Bbb Z$,
	$\Bbb Z/2\Bbb Z\times\Bbb Z/4\Bbb Z$ happen for  infinitely many  $m$, but  the case
	$\Bbb Z/2\Bbb Z\times\Bbb Z/8\Bbb Z$ occurs only for finitely many $m$.
\end{thm}

The following table shows examples for the torsion subgroups of $\mathcal G_\ta^m(\Bbb Q)$ respect to the conditions in Theorem~\ref{main5}.
We  refer the reader to  see  Table~\ref{tab:data} in the Appendix for more data on the rank and cardinal number of the torsion subgroup of $\mathcal G_\ta^m$ for some $0 < m < 1$.

The discriminant and $j$-invariant of the elliptic curve ${\mathcal G}_\ta^m$ are given respectively as follows:
\begin{align*}
\Delta_\ta^m &=16 r^2(r^2 - s^2)^4 (m + 1)^2 m^4\left(r^2m^2+2(2s^2-r^2)m+r^2\right),\\
j_\ta^m &=256\frac{\left(r^4m^4+4r^2s^2m^3+(s^4+6r^2s^2-r^4)m^2+4r^2s^2m+r^4\right)^3}{(m + 1)^2(r^2 - s^2)^4\left(r^2m^2+2(2s^2-r^2)m+r^2\right)m^4r^2}.
\end{align*}
It is  easy to check the following isomorphisms over rational numbers,
$${\mathcal G}_\ta^m\cong {\mathcal G}_\ta^{1/m}, \quad \text{and}\quad
{\mathcal G}_\ta^{-m}\cong {\mathcal G}_\ta^{-1/m}.$$
Thus,  we may assume the elliptic curves ${\mathcal G}_\ta^m$ with
$m\geq 1$ in the rest of paper.

We end this section by remarking that
${\mathcal{G}}_{{\pi}/{2}}^m$ is isomorphic to the elliptic curve
$$E(m): U^2=X(X-m^2)(X-(m-1)(m+1)),$$
by changing the variables $\mathcal X=X-m^2$ and  $\mathcal Y=U$, which is  given and studied in  \cite{Ochi}.
\begin{table}[htbp]
\caption{Examples for the torsion subgroups of $\mathcal G_\ta^m(\Bbb Q)$ in Theorem~\ref{main5}}
\label{tab:1}
		\begin{tabular}{llll}
			\toprule
			& $\{r,s,m\}$ & $\{\sqrt{M_0},\sqrt{M_1},\sqrt{M_2}\}$ & ${\mathcal G_\ta^m}(\Bbb Q)_{\text{tors}}$ \\ \midrule
			& $\{2,1,2\}$ & $\{\sqrt{6},2\sqrt{9+3\sqrt{6}},2\sqrt{9-3\sqrt{6}}\}$ & $\Bbb Z/4\Bbb Z$  \vspace{.2cm}\\
			& $\{2,1,3\}$ & $\{3,4\sqrt{7},4\}$ & $\Bbb Z/8\Bbb Z$ \vspace{.2cm}\\
			& $\{2,1,1\}$ & $\{\sqrt{3},2\sqrt{3}+2,2\sqrt{3}-2\}$ & $\Bbb Z/2\Bbb Z\times \Bbb Z/4\Bbb Z$ \vspace{.2cm}\\
			& $\{25,7,1\}$ & $\{24,70,10\}$ & $\Bbb Z/2\Bbb Z\times \Bbb Z/8\Bbb Z$ \\ \bottomrule
		\end{tabular}
\end{table}

\section{The main results related to  $\Nc_{\ta, m}$ and $\Nc_\ta$}
\label{Mainres}
Given  $\ta \in (0,\pi/2]$  with $\cos(\ta)=s/r \in \Q$ and  a rational number $m\geq 1$, we have the following results on $\Nc_{\ta, m}$.
\begin{thm}
	\label{main6}
	\begin{itemize}
		\item[(i)] If $\sqrt{m} \sin(\ta) $ and $\sqrt{m+1}\in \Q^*$, then
		$ \sqrt{m(r^2-s^2)} \ \bmod \ (\Q^*)^2$ belongs to $\Nc_{\ta, m}$,
		and some  points of   order~$8$  in  $\Gg_{\ta}^m(\Q)$ give some $\ta$-parallelogram envelope with ratio $m$.
		\item[(ii)] Otherwise, there exists a $\ta$-parallelogram envelope
		$(a, b, c, d, e)_\ta^m$	for  $n \in \Nc_{\ta, m}$
		if and only if  $ r_\ta(m) \geq 1.$
	\end{itemize}	
\end{thm}

Using the above theorem, we  obtain the next result.
\begin{thm}
	\label{main7}
	\begin{itemize}
		\item[(i)] The set $\Nc_{\ta, m}$ is empty if and only if  $\sqrt{m} \sin(\ta) $ or $\sqrt{m+1}$ is not a rational number 	and $ r_\ta(m)=0$.
		\item[(ii)]  $\Nc_{\ta, m}=\left\{\sqrt{m(r^2-s^2)} \ \bmod \  (\Q^*)^2 \right\}$ if and only if
		$\sqrt{m} \sin(\ta)$ and $\sqrt{m+1}$  belong to $\Q^*$ and  $ r_\ta(m)=0.$
		\item[(iii)] The set  $\Nc_{\ta, m} $  has  infinitely many elements if and only if   $ r_\ta(m) \geq 1.$
	\end{itemize}		
\end{thm}
\begin{thm}
	\label{main8}
	\begin{itemize}
		\item [(i)]
		A natural number $n$ belongs to $\Nc_{\ta, m}$ if and only if the  curve
		\[\Cc_\ta^{(m,n)}:	
		\left\{\begin{array}{l}
		\hspace{-.1cm} \Y^2= \X^3+(r^2m^2+2s^2m+r^2) \X^2+m^2(r^2-s^2)^2 \X \vspace{.2cm}\\
		\hspace{-.1cm} \Zc^2= \displaystyle \frac{2n (r^2-s^2)}{m+1}\left(\Y+ s (m+1) \X\right)\vspace{.1cm}\\
		\qquad \times \left(\X^2 -(r-s)(ms-r)\X - (r-s)\Y\right) \vspace{.2cm}\\
		\qquad \times \left(\X^2 +(r+s)(ms-r)\X + (r+s)\Y\right),\\
		\end{array}\right.\]
		has a rational point $(\X,\Y,\Zc)$ with $\Y \neq 0$.	
		\item [(ii)]  For any natural number $n$, there exist only finitely many $\ta$-parallelogram envelopes
		with ratio 	$m$ for $n$.
	\end{itemize}
\end{thm}

The case $\ta=\pi/2$ of the next result gives a positive answer for all parts of Question~\ref{ques1};
in other words, all questions in Section~7 of the T.~Ochai's paper.
\begin{thm}
	\label{main9}
	Given a Pythagorean angle $\ta \in (0, \pi/2]$, the following hold:
	\begin{itemize}	
		\item [(i)]  
	$\Nc_\ta=\N$.
		\item [(ii)] Given any  $n\in \Nc_\ta$, there are  infinitely many distinct $\ta$-parallelogram envelopes for $n$.
		\item[(iii)] Given any $n\in \Nc_\ta$,  there are infinitely many rational numbers $m\geq 1$ such that $n\in\Nc_{\ta, m}$.
	\end{itemize}
\end{thm}

We note that the proof of Theorem~\ref{main9}  works only for the Pythagorean angles. Hence, the interested reader may consider the following question as a further research on the subject.
\begin{ques}
	\label{ques2}
	Are   the statements of  Theorem~\ref{main9} true for any arbitrary angle
	$\ta \in (0, \pi/2]$ with rational cosine?
\end{ques}

\section{Proof of Theorem~\ref{main0}}
\label{Pr0}
Let $\cos(\tau)=p/q \in \Q$ with $p, q \in \Z$ satisfying  $0< p < q$ and
$\gcd(p,q)=1$.
By the assumption, we have
$\sqrt{r^2-s^2}=n' v^2 \sqrt{q^2-p^2}$, where $n'$ is a positive integer and $v$ is a rational number.

If $n\in\Nc_\ta(\tau)$, then $4n\in\Nc_\ta(\tau)$, and there exists a $\ta$-parallelogram envelope  $(a,b,c,d,e)_\ta^\tau$ with angle $\tau$ for $4n$, i.e.,
there exists a quintuple $(a,b,c,d,e)_\ta^\tau$ of positive rational  numbers satisfying
$$\begin{array}{ll}
\displaystyle c^2 =  a^2+b^2- \frac{2s ab}{r}, &     a(b+d)= 4 r n, \vspace{.2cm}\\
\displaystyle e^2 =  a^2+d^2+\frac{2s ab}{r}, & \displaystyle
\cos(\tau)=  \frac{a^2- bd - 4 s a (b-d)/r}{ce}.
\end{array}$$
These equations imply that
$$c^2+e^2 - \frac{2p c e}{q} = (b+d)^2.$$
Hence, we have  a rational $\tau$-triangle $(c,e, b+d)$ with area
$$ 2 n \sqrt{r^2-s^2}= 2 n n' v^2 \sqrt{q^2-p^2}.$$
Dividing all sides  by $v$ leads to a $\tau$-triangle $(c/v,e/v, (b+d)/v)$ with area
$2 n n' \sqrt{q^2-p^2}$, see Figure~6.
This means that $2 n n'$ is a $\tau$-congruent number as desired.
\begin{figure}[h]
		\begin{center}
		\includegraphics[height=.13\textheight,width=.52\textwidth]{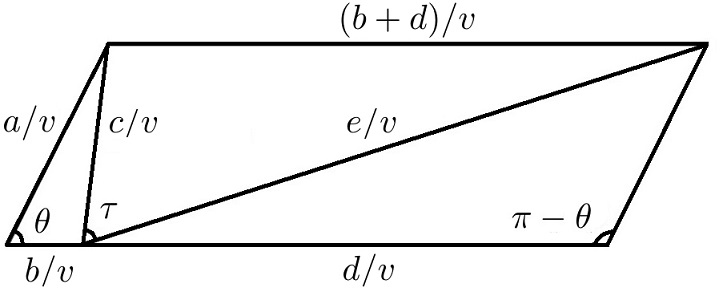}
		\caption{A $\tau$-triangle for $2n n'$ obtained from a $\ta$-parallelogram envelope for $n$}
	\end{center}
\end{figure}

Conversely, since $2n n'$ is a $\tau$-congruent number,  there exists a rational $\tau$-triangle $(a, b, c)$ for $2n n'$.
Then,  as described in the following,  one may find a $\ta$-parallelogram envelope with angle $\tau$ for $n$ given by
\begin{equation}
\label{e0}
\left( \frac{rab}{2qcn'},
\frac{a^2}{2c} - \frac{(pn'-s) ab}{2qcn'}, \frac{a}{2} ,\frac{c^2-a^2}{2c} +\frac{(pn'-s)ab}{2qcn'}, \frac{b}{2}\right)_\ta^\tau.
\end{equation}
We first note that the $\tau$-triangle $(a,b,c)$ for $2n n'$ satisfies
\begin{equation}
\label{0}
c^2 = a^2+b^2-\frac{2p ab}{q}  ,\quad  ab  = 4 q n n'.
\end{equation}
Hence, to find a  $\ta$-parallelogram envelope  with angle $\tau$ for $4n $, we  have to solve
\begin{align}
a^2 &= x^2+y^2-\frac{2sxy}{r}, \label{1}\\
b^2 &= x^2+(c-y)^2+\frac{2sx(c-y)}{r}, \label{2}\\
xc & = 4 r n. \label{3}
\end{align}
Substituting \eqref{1} into \eqref{2} and using \eqref{0}, we get $a^2-cy-sab/r+scx/r=0$
which gives
\begin{equation}
\label{4}
y = \frac{a^2}{c}-\frac{pab}{qc}+ \frac{sx}{r}.
\end{equation}
From \eqref{0} and \eqref{3},  we have  $x=rab/(n'qc)$ which implies
$y=a^2/c- (pn'-s)ab/(qcn')$ by applying \eqref{4}.
Therefore, we  have found the $\ta$-parallelogram envelope $(x, y, a, c-y, b)_{\ta}^{\tau}$ with angle $\tau$ for $4 n$. Dividing all components by $2$  leads to a $\ta$-parallelogram envelope  for  $n$ defined by \eqref{e0}, see Figure~7.
\begin{figure}[h]
		\begin{center}
		\includegraphics[height=.19\textheight,width=.62\textwidth]{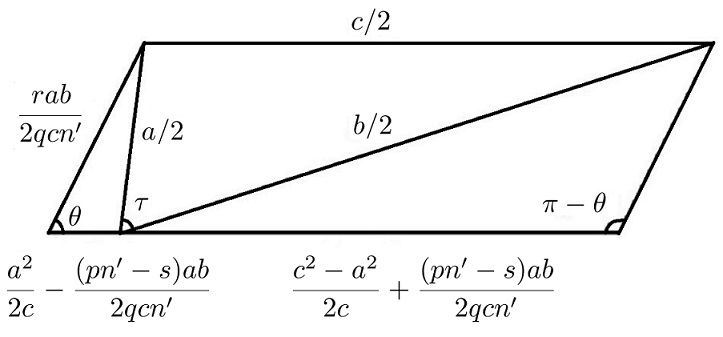}
		\caption{A $\ta$-parallelogram envelope with angle $\tau$ for $n$ obtained from a rational $\tau$-triangle for $2nn'$}
	\end{center}
\end{figure}
Checking that  it is really a $\ta$-parallelogram envelope  with angle $\tau$ for $n$ is left as an exercise to the reader.

For  $\tau=\pi/2$, we have a $\ta$-parallelogram envelope with angle $\pi/2$ given by
\begin{equation*}
\label{e1}	
\left(\frac{rab}{2cn'},
\frac{a^2}{2c}+\frac{s ab}{2cn'}, \frac{a}{2}, \frac{c^2-a^2}{2c}-\frac{s a b}{2cn'}, \frac{b}{2}\right)_{\pi/2}.
\end{equation*}
In the case $\tau=\ta$,  we obtain the following
$\ta$-parallelogram envelope with angle $\ta$ given by
\begin{equation*}
\label{e2}
\left(\frac{ab}{2c},
\frac{a^2}{2c}, \frac{a}{2}, \frac{c^2-a^2}{2c}, \frac{b}{2}\right)_\ta
\end{equation*}	
using a rational $\ta$-triangle $(a,b,c)$. The last statement is trivial by Remark~\ref{rem1}~(1) and  the definition of $\tau$-congruent numbers.
Therefore,  the proof of Theorem~\ref{main0} is completed.

\section{Proofs of Theorems~\ref{main1}, \ref{main2}, and \ref{main3}}
\label{Pr1-2}
The proof of Theorem~\ref{main1} is a consequence of the following one-to-one correspondence between $\Rc_\ta$, the set of all rational solutions of \eqref{eq5}, and  $\Pc_\ta$ the set of all $\ta$-parallelogram envelopes for $n$. Indeed, for any $(u, v, w, x, y)\in \Rc_\ta$, we have a $\ta$-parallelogram envelope for $n$ as follows:
\begin{equation}
\label{ee1}
\begin{split}
& a  = \displaystyle \left| \frac{y}{2x}\right| =\left| \frac{v}{2u}\right|,\quad  b  = \left| \frac{r w x}{y}\right|, \quad c  =  \left| \frac{x^2+(r^2-s^2)w^2}{2y}\right|, \vspace{.2cm}\\
& d  = \displaystyle \left| \frac{r (2n -w) u}{v}\right|, \quad e  =  \left| \frac{u^2+(r^2-s^2)(2n -w)^2}{2v}\right|.
\end{split}
\end{equation}
Conversely, any $\ta$-parallelogram envelope $(a, b, c, d, e)_\ta$ for $n$ corresponds to  a rational solution
$(u, v, w, x, y)$ defined by
\begin{equation*}
\label{ee2}
\begin{array}{lll}
x \displaystyle = 2 a \left(a+c-\frac{s}{r}b\right), & \displaystyle y= 4 a^2 \left(a+c-\frac{s}{r}b\right), & w=\displaystyle \frac{2 ab}{r}, \vspace{.2cm}\\
u \displaystyle = 2 a \left(a+e+\frac{s}{r}d\right), & \displaystyle v= 4 a^2 \left(a+e+\frac{s}{r}d\right).&
\end{array}
\end{equation*}

In order to show  Theorem~\ref{main2}, we first assume that  $n\in\Nc_\ta$. Then, by
Theorem~\ref{main1} there exists a rational solution
$(u,v,w,x,y)\in \Rc_\ta $ of Equations~\eqref{eq5} satisfying $yv \not =0$ and $0 < w \leq n$.
Substituting  $uy/x$ for $v$ into the equation of $F_\ta^N$ gives that
$$\frac{u^2 y^2}{x^2}= u(u+(r-s)N)(u-(r+s)N),$$
where $N=2n -w$.
Multiplying both sides of the latter equation by $x^2$ and then dividing by $u$, we obtain
$$x^2 u^2 - 2x^2Nsu- (r^2-s^2) N^2 x^2 - y^2 u =0.$$
Substituting  $y^2$ with $ x^3 + 2 s w x^2 -(r^2-s^2) w^2x$ into the above equation and then dividing both sides of the resulting equation by $x$ give a quadratic equation in terms of $u,$
$$x u^2 -(x^2 +2xs(N+w) -(r^2-s^2)w^2)u - (r^2-s^2)N^2 x=0.$$
Solving this equation, we obtain
\begin{align}
u&=\frac{x^2  + 2xs(N+w) -(r^2-s^2)w^2}{2x}\nonumber\\
&\quad \pm\frac{\sqrt{(x^2 + 2xs(N+w) -(r^2-s^2)w^2)^2+
		4(r^2-s^2)N^2 x^2} }{2x}. \label{eq7}
\end{align}
Since $x$, $u$, and $w$ are all rational, there exists  some $z \in \Q$ such that
\begin{equation}
\label{eq8}
z^2=(x^2 + 2xs(N+w) -(r^2-s^2)w^2)^2+ 4(r^2-s^2)N^2 x^2.
\end{equation}
Thus, we obtain the desired rational solution $(w, x, y, z)$ of Equation~\eqref{eq6}.
Conversely, if we assume that $(w, x, y, z)$ with $y \not =0$ and $0<w \leq n $ is a solution
of Equation~\eqref{eq6}, then we obtain a rational solution
$(u, v, w, x, y)$ of \eqref{eq5}, where $u$ is given by Equation~\eqref{eq7} and $v= yu/x \not =0$, respectively.
Therefore,  the proof of Theorem~\ref{main2} is completed.

By the condition $d=mb$,  we obtain that
$$ a ( b+ d)=rn \Longleftrightarrow a b(m+1)=rn \Longleftrightarrow \frac{2a b}{r}=\frac{2n}{m+1},$$
and the aforementioned correspondence with $w=2n/(m+1)$  implies that a natural number $n$ belongs to $\Nc_{\ta,m}$  if and only if the simultaneous Equations~\eqref{eq5} with $N=2nm/(m+1)$ have a rational solution $(u,v,x,y)$ with $y v \not =0$.
Now, using Theorem~\ref{main2} and some simple algebraic simplifications on Equation~\eqref{eq8} we obtain the desired Equation~\eqref{eq6a}. This completes the proof of Theorem~\ref{main3}.

\section{Proofs of Theorems~\ref{main4} and \ref{main5}}
\label{main4-5}
We fix  the natural number $n$,  the rational number $m\geq 1$ and an angle $\ta \in (0,\pi/2]$ with rational cosine.
One can transform the quartic $G_\ta^{(m,n)}$
into the cubic $\Gg_\ta^{m}$ by the following birational transformations:
\begin{equation}
\left\{
\begin{array}{l}
\label{trans1}
\hspace{-.1cm} x\displaystyle=\frac{-2n (r^2-s^2) \left(\Y + s(m+1) \X\right)}{(m+1) \X (\X + r^2-s^2)},\vspace{.2cm}\\
\hspace{-.1cm} z\displaystyle=\frac{4 n^2 (r^2-s^2)}{(m+1)^2 \X (\X + r^2-s^2)^2}\big(\X^3 + (2 d_2 -(r^2-s^2)) \X^2 \vspace{.1cm}\\
\qquad + \left(3 d_1+ 2s (m+1) \Y\right) \X + (r^2-s^2) d_1\big),
\end{array}
\right.
\end{equation}
where $  d_1=m^2(r^2-s^2)^2$ and  $d_2=(r^2m^2+2s^2m+r^2)$ are coefficients of the elliptic curve
${\mathcal G}_\ta^m$, and
\begin{equation*}
\left\{
\begin{array}{l}
\label{trans2}
\hspace{-.1cm} \X=\displaystyle\frac{- (r^2-s^2) \left((m+1)^2 x^2 + 4 sn(m+1)^2 x + (m+1)^2 z - 4 n^2 (r^2-s^2)\right)}{2 (m+1)^2 x^2},\vspace{.2cm}\\
\hspace{-.1cm} \Y=\displaystyle\frac{(r^2-s^2) \left( s (m+1)^4 x^3 + c_2 x^2 + c_1 x  +c_0\right) }{2 (m+1)^3 x^3},
\end{array}
\right.
\end{equation*}
where
\begin{align*}
c_0 & =- 2n (r^2-s^2)\left( (m+1)^2 z - 4 n^2 (r^2-s^2)\right), \\
c_1 & = s (m+1)^2 \left( (m+1)^2 z - 12 n^2 (r^2-s^2)\right), \\
c_2 & = 2 n (m+1)^2 \left( 2 d_2 - 3 (r^2-s^2)\right).
\end{align*}

The points $(x,z)=\left( 0,\pm 2n(r^2-s^2)(m+1) \right) $ on the quartic $G_\ta^{(m,n)}$ map to the point at infinity $\infty=[0:1:0]$ on ${\mathcal G}_\ta^m$ by the above change of variables.
The  Mordell-Weil group of rational points on ${\mathcal G}_\ta^m$ has generic rank one, because
one can easily examine by {\sf{SAGE}} software that   $[n]\Pc $ are not $\infty$ for $2 \leq n\leq 16$, where
$$\Pc=(-(r^2-s^2) m^2, s (r^2-s^2) m^2  (m+1)).$$
Hence, the point $\Pc$ is of infinite order on ${\mathcal G}_\ta^m(\Q)$.
This completes the proof of Theorem~\ref{main4}.

In what follows we classify the possibilities for ${{\mathcal G}_\ta^m(\Bbb Q)}_{\text{tors}}$.
It has trivially the point $(0,0)$, which is of order~$2$.
Hence, by the celebrated Mazur's theorem \cite{slm1} on the torsion subgroup of elliptic curves over the rational numbers, we have  the following possibilities:
$$\frac{\Z}{2\Z}, \ \frac{\Z}{4\Z}, \ \frac{\Z}{6\Z}, \ \frac{\Z}{8\Z}, \ \frac{\Z}{10\Z}, \
\frac{\Z}{12\Z}, \ \frac{\Z}{2\Z}\times \frac{\Z}{4\Z}, \ \frac{\Z}{2\Z}\times \frac{\Z}{6\Z}, \
\frac{\Z}{2\Z}\times \frac{\Z}{8\Z}.$$
For any point $\Tc=(\mathcal{X},\mathcal{Y})$ on the curve $\mathcal{G}_\ta^m$, the duplication formula of ${\mathcal G}_\ta^m(\Bbb Q)$ leads to
\begin{align}
\label{Dup}
{\mathcal X}([2] \Tc)&= \frac{\left( {\mathcal X}^2-(r^2-s^2)^2m^2 \right)^2}{{4 \mathcal Y}^2}.
\end{align}
Hence the $\mathcal X$-coordinates of the points $[2n]\Tc$, for $n=2, 3, 4, 5, 6$,  are also squares.

If $\Tc$ is a point of order~$4$, then  ${\mathcal X}([2]\Tc)=0$.
By the duplication formula  and the fact that $(0,0)$ is the unique  point of order~$2$ on $\Gg_\ta^m$,  we find out two points of order~$4$ as $\Tc$ and $-\Tc$, where
$$\Tc:=\left((r^2-s^2)m, r(r^2-s^2)m(m+1)\right).$$
Thus,  the cases $\Bbb Z/2\Bbb Z$, $\Bbb Z/6\Bbb Z$ and $\Bbb Z/10\Bbb Z$ cannot happen. Hence, the  possibilities for torsion subgroup reduce to
$$\frac{\Z}{4\Z}, \ \frac{\Z}{8\Z}, \  \frac{\Z}{12\Z}, \ \frac{\Z}{2\Z}\times \frac{\Z}{4\Z}, \ \frac{\Z}{2\Z}\times \frac{\Z}{6\Z}, \ \frac{\Z}{2\Z}\times \frac{\Z}{8\Z}.$$
Let us rule out the two possibilities $\Bbb Z/12\Bbb Z$ and $\Bbb Z/2\Bbb Z\times\Bbb Z/6\Bbb Z$.
To this end, it suffices to show the non-existence of any  point of order~$3$.
By contrary,  we  assume that $\Tc=(\X, \Y)$ is such a point. Then, the equality $[2]\Tc= -\Tc$  implies that
$$-3{\mathcal X}^4+(-4r^2-8ms^2-4m^2r^2){\mathcal X}^3-6m^2(r-s)^2(r+s)^2{\mathcal X}^2+m^4(r-s)^4(r+s)^4=0,$$
which is unsolvable over $\Q$. Hence, the cases $\Z/12\Z$ and $\Z/2\Z\times\Z/6\Z$ can never happen as torsion subgroup of  ${\mathcal G}_\ta^m$ over $\Q$.
Considering these observations  and applying the Mazur's theorem show that
$${\mathcal G}_\ta^m(\Q)_{\text{tors}}\cong \frac{\Z}{4\Z},~\text{or}~\frac{\Z}{8\Z},~\text{or}~
\frac{\Z}{2\Z}\times \frac{\Z}{4\Z}~\text{or}~\frac{\Z}{2\Z}\times \frac{\Z}{8\Z}.$$

Now, if we assume that $\Tc$ is a point of order~$8$, then $[2]\Tc$ (resp. $[4]\Tc$) will be a point of order~$4$ (resp. $2$). Solving ${\mathcal X}([2]\Tc)=(r^2-s^2)m$ is equivalent to finding the solutions of the following quartic equation,
\begin{align*}
& {\mathcal X}^4-4m(r^2-s^2){\mathcal X}^3-2m(r^2-s^2)(2r^2+mr^2+2m^2r^2+3ms^2){\mathcal X}^2\\
&\quad -4m^3(r^2-s^2)^3{\mathcal X}+m^4(r^2-s^2)^4=0,
\end{align*}
which can be rewritten as
$({\mathcal X}-m(r^2-s^2))^4=4mr^2(r^2-s^2)(m+1)^2{\mathcal X}^2$. Solving this equation leads to the following list of $\X$-coordinates of $\Tc$,
\begin{equation*}
\begin{array}{ll}
\label{XX}
\X_1=M_0+\left(r(m+1)+\sqrt{M_1}\right)\sqrt{M_0},\vspace{.2cm} \\
\X_2=M_0+\left(r(m+1) -\sqrt{M_1}\right)\sqrt{M_0},\vspace{.2cm}\\
\X_3=M_0-\left(r(m+1)+\sqrt{M_2}\right)\sqrt{M_0}, \vspace{.2cm}\\
\X_4=M_0-\left(r(m+1)-\sqrt{M_2}\right)\sqrt{M_0},
\end{array}
\end{equation*}
where $M_0, M_1$, and $M_2$ are given by \eqref{Ms}.
Using the duplication formula~\eqref{Dup}, one may write down the $\Y$-coordinates as follow:
\begin{equation*}
\begin{array}{l}
\label{YY}
\Y_1 =\frac{1}{2}M_0\left(r(m+1)+\sqrt{M_1}\right)\left( r(m+1)+\sqrt{M_1} +2\right), \vspace{.2cm}\\
\Y_2=\frac{1}{2}M_0\left(r(m+1)-\sqrt{M_1}\right)\left( r(m+1)-\sqrt{M_1} +2\right), \vspace{.2cm}\\
\Y_3=\frac{1}{2}M_0\left(r(m+1)+\sqrt{M_2}\right)\left( r(m+1)+\sqrt{M_2} -2\right), \vspace{.2cm}\\
\Y_4=\frac{1}{2}M_0\left(r(m+1)-\sqrt{M_2}\right)\left( r(m+1)+\sqrt{M_2} -2\right).
\end{array}\end{equation*}

Thus, in order to have ${\mathcal G}_\ta^m(\Q)_{\text{tors}}\cong \Z/2\Z\times\Z/8\Z$, we need to assume that all the numbers $\sqrt{M_0}, \sqrt{M_1},$ and $\sqrt{M_2}$ are rational.
If one of the values  $\sqrt{M_1}$ and $ \sqrt{M_2},$ does not belong to $\Q$, then
${\mathcal G}_\ta^m(\Q)_{\text{tors}}\cong \Z/8\Z$ provided that $\sqrt{M_0} \in \Q^*$.
When the last condition fails to be true, we will have
${\mathcal G}_\ta^m(\Q)_{\text{tors}}\cong  \Z/4\Z \ \text{or} \ \Z/2\Z \times \Z/4\Z.$

In order to prove the last statement of Theorem~\ref{main5}, we assume that $M_0=m_0^2$ is a rational square. Then
$m = m_0^2/(r^2-s^2),$
which yields
\begin{align*}
M_1&= \frac{r(m_0^2+r^2-s^2)\left(r(m_0^2+(r^2-s^2))+2m_0(r^2-s^2)\right)}{(r^2-s^2)^2},
\\
M_2&= \frac{r(m_0^2+r^2-s^2)\left(r(m_0^2+(r^2-s^2))-2m_0(r^2-s^2)\right)}{(r^2-s^2)^2}.
\end{align*}
We let	$m_i=m_0^2 (r^2-s^2) \sqrt{M_i}$ for $i=1,2.$ Then,
\begin{align}
m_1^2&=r(m_0^2+r^2-s^2)\left(r(m_0^2+(r^2-s^2))+2m_0(r^2-s^2)\right), \label{F}\\
m_2^2&=r(m_0^2+r^2-s^2)\left(r(m_0^2+(r^2-s^2))-2m_0(r^2-s^2)\right). \label{G}
\end{align}
Each of Equations~\eqref{F} and \eqref{G} defines an elliptic curve over $\Q$, which is birational to
$$E_0: Y^2=X^3-108\,{r}^{2} ({r}^{2}-{s}^{2})^2 ({r}^{2}+3\,{s}^{2}) X+
432\,{r}^{4}  (r^2-s^2)^3 (r^2-9 s^2), $$
by the following change of variables:
\begin{equation*}
\left\{
\begin{array}{l}
\label{trans3}
\hspace{-.1cm} X =\displaystyle\frac{6r\left(r(r^2-s^2)m_0^2 +3 \varepsilon  (r^2-s^2)m_0 +3 (m_i+r) \right)}{m_0^2},\vspace{.2cm}\\
\hspace{-.1cm} Y =\displaystyle\frac{54  r  (r^2-s^2) m_0 \left(r(r^2-s^2)^2 m_0^2 + 2 \varepsilon r^2 m_0 + (m_i +3r) \right) +2  \varepsilon r (m_i+r) }{m_0^3},
\end{array}
\right.
\end{equation*}
and
\begin{equation*}
\left\{
\begin{array}{l}
\label{trans4}
\hspace{-.1cm} m_0\displaystyle=\frac{6r\left(Y + 3 \varepsilon (r^2-s^2)(X + 12r^2(r^2-s^2)\right)}
{\left( X + 12 r^2(r^2-s^2)\right) \left(X -24 r^2(r^2-s^2)\right)},\vspace{.2cm}\\
\hspace{-.1cm} m_i\displaystyle=\frac{r\left(X^3 + e_1 X  + e_0\right)}
{\left( X + 12 r^2(r^2-s^2)\right) \left(X -24 r^2(r^2-s^2)\right)^2},
\end{array}
\right.
\end{equation*}
where  $\varepsilon=(-1)^{i-1}$ for $i=1,2,$ and
\begin{align*}
e_0  &=- 72 r^2  (r^2-s^2)^2 \left( 48 r^4 (r^2-s^2) + \varepsilon Y \right),\\
e_1 &= 108 r^2 (r^2-s^2)(5r^2-9s^2) \left((r^2-s^2)+ \varepsilon Y \right).
\end{align*}
The  torsion subgroup of $E_0$ is  $\Z/2\Z \times \Z/2\Z$  containing
the   point at infinity and
$$P_0=(-12 r^2(r^2-s^2),0), \ P_1=(-6 r (r-s)(r^2-9 s^2),0),\
P_2=(-6 r (r+3s) (r^2-s^2),0).$$
Moreover, the Mordell-Weil rank of $E_0 (\Q)$ is at least one with the point of infinite order
$$Q=\left( -3 (r^2-s^2) (r^2 +3 s^2), 27 (r^2-s^2)^3\right).$$
By the above change of variables, the point $(m_0, m_i)=(0, r (r^2-s^2))$  on Equations~\eqref{F} and \eqref{G} gives the point at infinity on $E_0$, and the points $P_0$ and $Q-P_0$ give two points at infinity on the quartic curves~\eqref{F} and \eqref{G}.  For a positive integer $\ell$,  the  $\ell$-multiplication $[\ell] Q$  of  $Q$ gives either the rational solution $(m_0, m_1)$
or $(m_0, m_2)$. Hence, each of  $\sqrt{M_1}$ and $\sqrt{M_2}$ can be  rational numbers infinitely many times, when the multiplications of $Q$ vary.
Moreover, the torsion points  $P_1$ and $ P_2$ are the only ones that give a number $m_0$  leading to the rational numbers $m_1$, $m_2$ and hence the rational numbers $\sqrt{M_1}$ and $\sqrt{M_2}$.
This argument justifies the last statement of Theorem~\ref{main5}.

\section{Proofs of the main results on $\Nc_{\ta,m}$}
\subsection{Proof of Theorem~\ref{main6}}
Let  $\sqrt{m+1}=v $ and $\sqrt{m} \sin(\ta)= N'  u^2  $,   where
$N'$ is a positive integer and $\alpha, u\in \Q$.
Then,  $\sqrt{M_0} =  N' r  u^2$ and
$$\sqrt{M_1} = \frac{r v}{u} \sqrt{(v/u)^2+ 2 N'},  \quad
\sqrt{M_2} = \frac{r v}{u} \sqrt{(v/u) ^2- 2 N'}.$$
Hence, both the numbers $\sqrt{M_1}$ and $\sqrt{M_2}$ are rational if and only if   there are rational numbers $\alpha$ and $\beta$ such that $(v/u)^2+ 2 N' =\alpha^2$ and $ (v/u) ^2- 2 N'=\beta^2,$ which is equivalent to saying that $2N'$ is a congruent number with a right triangle
$(\alpha-\beta, \alpha+\beta, 2 \gamma)$, where $\gamma= v/u$.
Hence, for  such a right triangle we have
$\sqrt{M_1}= r \alpha \gamma $ and $\sqrt{M_2}= r \beta \gamma.$
In order to prove the part~(i) of Theorem~\ref{main6}, one needs to  do cumbersome computations using a series of transformations which we only give a sketch and leave the detailed computations to the reader.
For $i=1,\ldots, 4$,  writing the order~$8$ points $(\X_i, \pm \Y_i)$ in terms of $N', u, v, \alpha, \beta$,    transforming  them into the points $(x_i, z_i)$ on $G_\ta^{(m,n)}$ by  \eqref{trans1}, and  finding
$y_i$-coordinates on $E_\ta^{(m,n)}$ using $x_i$'s, we obtain rational solutions $(x_i, y_i, z_i)$ of
Equation~\eqref{eq6a}. Calculating $u_i$  from $x_i$ and $z_i$ using the formula~\eqref{eq7} and then putting
$v_i=u_i y_i/x_i$ lead to the solutions $(u_i,v_i, w, x_i, y_i )$ with $w=2 N'/(N' u^2+1)$ of  Equations~\eqref{eq6}. Now, using the correspondence~\eqref{ee1}, one may obtain a
$\ta$-parallelograms envelope with ratio $m$ for $\sqrt{M_0}\ \bmod \ (\Q^*)^2$, which means that it belongs to $\Nc_{\ta, m}$ as desired
in the part~(i) of Theorem~\ref{main6}.

If one or both $\sqrt{m+1}$ and $\sqrt{m} \sin(\ta)$ do not belong to $\Q$, then at
least one of the  $\sqrt{M_i}$ will be a non-rational number, hence the torsion subgroup of $\Gg_\ta^m$ cannot have a  point of order~$8$. Thus, it will be isomorphic to $\Z/4\Z$ or $\Z/2 \Z \times \Z/4\Z$, by Theorem~\ref{main5}. The order~$4$ points in $\Gg_\ta^m(\Q)_{tors}$ are
$\pm \Tc=((r^2-s^2)m, \pm r(r^2-s^2) m (m+1))$, which lead to a rational point $(x,z)$ on $G_\ta^{(m,n)}$ with $x$-coordinate $-2n (r+s)/(m+1)$. But, this is a zero of the cubic polynomial defining $E_\ta^{(m,n)}$.
Thus no points of order~$4$ in $\Gg_\ta^m(\Q)_{\rm{tors}}$ give us a $\ta$-parallelogram envelope.
Therefore,  the existence of a $\ta$-parallelogram envelope with ration $m$ for $n$, which one or both $\sqrt{m+1}$ and $\sqrt{m} \sin(\ta)$ are not rational, means that $\Gg_\ta^m(\Q)$ has a point of infinite order.

Conversely, we let $m$ and $\theta$ be as before, and  $n$ be a natural number.
Given any point  $(\X,\Y) \in \Gg_\ta^{(m,n)} (\Q)$ of infinite order,
the $\Y$-coordinates cannot be either
$0$  or $ \pm r(m+1) \X$ since it is neither of order~$2$ nor $4$.
By substituting
$$x=\frac{-2n (r^2-s^2) (\Y + s (m+1) \X)}{ (m+1) \X (\X + r^2-s^2)}$$
one can  write  the defining equation of $E_\ta^{(m,n)}$ as
$$y^2=\left( \frac{2n (r^2-s^2)}{(m+1) \X (\X + r^2-s^2)}\right)^2 \cdot 	A_\ta^{(m,n)}, $$
where
\begin{align*}
A_\ta^{(m,n)}&=\frac{2n\left(\Y+ s (m+1)\X\right)\left(\X^2 -(r-s)(ms-r)\X - (r-s)\Y\right)}{(m+1) \X (\X + r^2-s^2)} \nonumber\\
&\quad \times \left(\X^2 +(r+s)(ms-r)\X + (r+s)\Y\right).  \label{A1}
\end{align*}

If it is necessary, one can choose the sign of $\Y$ in a way such that $A_\ta^{(m,n)}$ is a positive rational number. To see this, we may assume $0<s<r$.
Since $(\X,\Y)$ is not of finite order and   $\Gg_\ta^m$ passes through the origin of the coordinate system, we have $\X>0$ implying that $\Y+ s (m+1) \X>0$.
Hence, the expression $A_\ta^{(m,n)}$ is positive if and only if the following two expressions are of the same sign,
$$A_1:=\X^2-(r-s)(ms-r)\X-(r-s)\Y,\quad A_2:=\X^2+(r+s)(ms-r)\X+(r+s)\Y.$$
For $\Y>0$ we have two cases:
\begin{itemize}
	\item[(i)]  $A_1, A_2>0$:  The relation $A_1>0$ holds if and only  if
	$$\X>(r-s)(ms-r) \quad \text{and} \quad \Y<\frac{\X(\X-(r-s)(ms-r))}{r-s},$$
	and $A_2>0$ if and only if
	$$\X>(r+s)(r-ms) \quad \text{and} \quad \Y>0$$
	or
	$$\X>(r+s)(r-ms)\quad \text{and} \quad \Y>-\frac{\X(\X+(r+s)(ms-r))}{r+s}.$$
	\item[(ii)] $A_1, A_2<0$:  The relation $A_1<0$ holds if and only if
	$$\X<(r-s)(ms-r) \quad \text{and} \quad \Y>0$$
	or
	$$\X> (r-s)(ms-r)\quad \text{and} \quad \Y>\frac{\X(\X-(r-s)(ms-r))}{r-s},$$
	and $A_2<0$ if and only if
	$$\X<(r+s)(r-ms) \quad \text{and} \quad \Y<-\frac{\X(\X+(r+s)(ms-r))}{r+s}.$$
\end{itemize}
For $\Y<0$ we have two cases:
\begin{itemize}
	\item [(i)]  $A_1, A_2>0$:  The relation $A_1>0$ holds if and only if
	$$\X>(r-s)(ms-r) \quad \text{and} \quad \Y<0$$
	or
	$$\X< (r-s)(ms-r)\quad \text{and} \quad \Y < \frac{\X(\X-(r-s)(ms-r))}{r-s},$$
	and $A_2>0$ if and only if
	$$\X>(r+s)(r-ms) \quad \text{and} \quad \Y>-\frac{\X(\X+(r+s)(ms-r))}{r+s}.$$
	\item [(ii)]  $A_1, A_2<0$.  The relation $A_1<0$ holds if and only if
	$$\X<(r-s)(ms-r) \quad \text{and} \quad \Y>\frac{\X(\X-(r-s)(ms-r))}{r-s}$$
	and $A_2>0$ if and only if
	$$\X<(r+s)(r-ms) \quad \text{and} \quad \Y<0$$
	or
	$$\X >(r+s)(r-ms) \quad \text{and} \quad \Y < -\frac{\X(\X+(r+s)(ms-r))}{r+s}.$$
\end{itemize}

Therefore, by Theorem~\ref{main2}, a natural number  $n\in \N$  belongs to  $\Nc_{\ta, m}$ provided that
$A_\ta^{(m,n)}$ is a rational square number.

\subsection{Proof of Theorem~\ref{main7}}
The parts~(i) and  (ii) are  direct consequences of Theorem~\ref{main6}.
The necessity of  (iii) is clear by sufficiencies of (i) and (ii).

For the sufficiency of (iii), we assume that $r_\ta(m)\geq 1$ and $(\X,\Y) $ is a point of infinite order.
Let $\varphi: \Gg_\ta^{(m,n)} (\Q) \rightarrow \Q$ be defined by  $\varphi(\infty)=0$ and
\begin{align*}
\varphi(\X,\Y)&= A_\ta^{(m,n)} \cdot \left( \frac{m+1}{2n}\right)  \X (\X+ r^2-s^2)\\
&= (\Y+ s (m+1) \X)\left(\X^2 -(r-s)(ms-r)\X - (r-s)\Y\right)\\
&\quad \times \left(\X^2 +(r+s)(ms-r)\X + (r+s)\Y\right).
\end{align*}
We also define the map $\pi: \Q \rightarrow \Q^*/(\Q^*)^2 \cup \{0\}$ by
$$\pi(q)=
\left\{\begin{array}{ll}
\hspace{-.1cm}  q \ \bmod \ \Q^* &  \text{\rm if}  \ q \neq 0 \vspace{.2cm}\\
\hspace{-.1cm} 0  &  \text{\rm if} \ q =0.
\end{array}\right.$$
Then, the set $\Nc_{\ta,m}$ has infinite element if and only if the image of $\Gg_\ta^{(m,n)} (\Q)$ under the composite map   $\pi \circ \varphi $ is infinite by the proof of the part~(ii) of Theorem~\ref{main6}.
We define $\Sigma$ to be the following set
$$\left\{ \Q(\varphi(P)): \ P=(\X,\Y)\in \Gg_\ta^{(m,n)} (\Q)\right\}.$$
Then,  the image of $\Gg_\ta^{(m,n)} (\Q)$ by   $\pi \circ \varphi $ is infinite if and only if $\Sigma $ is infinite.
Because, if we assume that $\Sigma $ is finite, then  by the pigeonhole principal there is a field ${\Bbb K} \in \Sigma$ and an infinite subset  $H \subset \Gg_\ta^{(m,n)} (\Q)$
such that ${\Bbb K} =\Q (\varphi (P))$ for all $P\in H$.
This implies that the algebraic surface $\Sc$ defined by the equation
$ \Zc^2=\varphi(\X,\Y)$
has infinitely many ${\Bbb K}$-rational points
$(\X,\Y,\Zc)=(\X(P), \Y(P), \sqrt{\varphi(P)})$ for $P\in H$.
Furthermore, $P\in \Gg_\ta^{(m,n)} (\Q)$ so that the following algebraic curve
$$\Cc_\ta^m:
\left\{\begin{array}{l}
\hspace{-.1cm} \Zc^2= \varphi(\X,\Y), \vspace{.2cm}\\
\hspace{-.1cm} \Y^2= \X^3+(r^2m^2+2s^2m+r^2) \X^2+m^2(r^2-s^2)^2 \X,
\end{array}\right.
$$
has also infinitely many ${\Bbb K}$-rational points as above.
But, this is a contradiction with the  Faltings' theorem \cite{slm3} on the  set of rational points on algebraic curves with genus~$\geq 2$, since  we have the following result:
\begin{lema}
	\label{genus}
	The  genus of  $\Cc_\ta^m$ is equal to  $9$.
\end{lema}
\begin{proof}
	Let  $\omega$  be  the canonical sheaf of $\Cc_\ta^m$.
	By Exercise (II.8.4.e) in \cite{hart}, it is isomorphic to the invertible sheaf $ \Oo(4)$,
	and hence	$\deg (\omega)= 4 \cdot \deg(\Cc_\ta^m).$
	Using the classical version of the Bezout's theorem, see Proposition 8.4 in \cite{fulton1} or Example 1 on page 198 of \cite{shaf1}, the degree of  $\Cc_\ta^m$ over  $\CC$ is equal to $4$ and so
	$\deg (\omega)= 16$.
	As a consequence of the Riemann-Roch theorem, 	it is well known that the degree of canonical sheaf of
	any algebraic curve of genus~$g$ is equal to $2g-2$, see Example 1.3.3 in Chapter IV of \cite{hart}.
	Therefore, the genus of  $\Cc_\ta^m$ is equal to  $9$.
\end{proof}

\subsection{Proof of Theorem~\ref{main8}}

Putting $\Zc=(m+1) \X (\X + r^2-s^2) y/(2n(r^2-s^2))$ and using the variable changes given by \eqref{trans1}, the simultaneous
equations defining $E_\ta^{(m,n)}$ and $G_\ta^{(m,n)}$ give the space curve $\Cc_\ta^{(m,n)}$ given in the statement of Theorem~\ref{main8}.
On the other hand, it is easy to  check that there exists a solution $(x,y,z)$  of Equation~\eqref{eq6a}  if and only if $\Cc_\ta^{(m,n)}$  has a solution $(\X,\Y,\Zc)$ with $\Y \neq 0$.
This shows the part~(i) of Theorem~\ref{main8}. By a similar argument as given in the proof of Lemma~\ref{genus},   one can  show that  $\Cc_\ta^{(m,n)}$ is an algebraic curve of  genus~$9$.
Hence,  it contains only finitely many rational solutions. Therefore, applying Theorem~\ref{main3} gives us the part~(ii) of  Theorem~\ref{main8}.

\subsection{Proof of Theorem~\ref{main9}}
Since we have supposed that $\ta$ is a Pythagorean angle with $\cos(\ta)=s/r$, its sine is rational number, i.e.,  $\sin(\ta)= \sqrt{r^2-s^2}/r \in \Q^*$.
First, we prove the  part~(i) by  the following two steps.
\bigskip

{\bf  Step 1:} {\it Given a natural number $n$, if $(c, e, f)$ is a rational triangle with area
	$T=2 n \sqrt{r^2-s^2} $,  then
	there exists a $\ta$-parallelogram envelope  $(a/2, b/2, c/2, d/2, e/2)_{\ta}$ for $n$, where
\begin{equation}
\label{eqt-1}
a  =\frac{4 r n}{f},\quad
b   = \frac{f^2 + c^2 - e^2}{2f} +  \frac{a s}{r}, \quad
d  = \frac{f^2 + e^2- c^2}{2f} -  \frac{a s}{r}
\end{equation}}

Without losing the generality, we may assume that $c< e <f$ holds for the sides of the triangle.
In order to see that the quintuple $(a, b, c, d, e)$ is $\ta$-parallelogram envelope for
 $4 n$, it is enough to check that Equations~\eqref{eq4} hold.
It is clear that $a+b=f$ and  hence $a(b+d)=a f = 4 r n$, which is the last equality in \eqref{eq4}.
Substituting $b$ and $d$, as given above, into the first and second equations of \eqref{eq4}
shows that
\begin{align*}
a^2+b^2  -c^2 - \frac{2 s a b}{r} &=a^2+d^2  -e^2 + \frac{2 s a d}{r}\\
&=64 n^2 (r^2-s^2)+(c+e+f)(c+e-f)(c-e-f)(c-e+f).
\end{align*}
Letting $p=(c+e+f)/2$ and  using the Heron's formula, we have
\begin{align*}
2n \sqrt{r^2-s^2} & = \sqrt{p(p-c)(p-e)(p-f)}\\
&=\frac{1}{4}\sqrt{(c+e+f)(c+e-f)(e+f-c)(c-e+f)},
\end{align*}
in  other words,
$$64 n^2 (r^2-s^2)+(c+e+f)(c+e-f)(c-e-f)(c-e+f)=0.$$
Hence, the quintuple $(a, b, c, d, e)$ is $\ta$-parallelogram envelope for $4 n$.
Dividing by $2$, we obtain   $(a/2, b/2, c/2, d/2, e/2)_{\ta}$ for $n$ as desired.
\bigskip

{\bf  Step 2:} {\it For any  positive rational number $T$ there exists a rational triangle  with area $T$.}

Indeed, this fact is proved by N.~J.~Fine in  Theorem~2 of \cite{fine}.
 Let us consider  the following  genus~one quartic curve
\begin{equation*}
\label{eqt3}
C_T: y^2= T^2 x^4+ T^2 x^3 -x -1.
\end{equation*}
Then, for any rational point $P=(x,y)$  with nonzero coordinates, we have a rational triangle with area $T$ given by
\begin{equation*}
\label{eqt1}
(c, e, f)=\left(\frac{y}{x},  \frac{T^2 x^2 +1}{y},
 \frac{T^2 x^4 +1}{x y} \right).
\end{equation*}
Indeed, if we let $\tau$ be the angle between  $c$ and $e$, then  using the equation of $C_T$ we get that
$$
\cos(\ta)= \frac{c^2+e^2-f^2}{2 c e}={\frac {{T}^{2}{x}^{2}-1}{{T}^{2}{x}^{2}+1}},\quad
\sin(\ta)={\frac {2 xT}{{T}^{2}{x}^{2}+1}},
$$
and hence  the area of triangle $(c, e,f)$  is equal to
$ c e \sin(\ta)/2=T,$ as desired.

Now, clearly we have  $\Nc_{\ta} \subseteq \N$. On the other hand, given any natural number $n$ and any Pythagorean angle $\ta \in (0,\pi/2]$ with $\cos(\ta)=s/r$ and $\sin(\ta) \in \Q^*$, letting  $T=  2 n \sqrt{r^2-s^2}$ in the above arguments leads to a  rational triangle $(c, e, f)$ and hence
$\ta$-parallelogram  envelope  $(a/2, b/2, c/2, d/2, e/2)_{\ta}$ for $n$ by Step~1.
Thus, the proof of part~(i) is completed.

In order to prove the part~(ii), we shall to investigate a little bit more on the curve $C_T$.
Clearly, it contains the rational point  $P_0=(-1,0)$ and hence
one can easily check that it  is birational to the elliptic curve,
\begin{equation*}
\label{eqt4}
E_T: Y^2= X^3+ 3 T^2 X - T^2 (T^2-1),
\end{equation*}
by the following  maps
\begin{equation*}
\label{eqt5}
\left\{
\begin{array}{ll}
\hspace{-.1cm} x =\displaystyle\frac{X+1}{T^2-X}, &  y=\displaystyle\frac{(T^2+1)Y}{(T^2-X)^2} \vspace{.2cm}\\
\hspace{-.1cm} X =\displaystyle\frac{T^2 x- 1}{x+1}, & Y=\displaystyle\frac{(T^2+1)y}{(x+1)^2}.
\end{array}
\right.
\end{equation*}
We note that the point $P_0=(-1,0)$ on $C_T$ corresponds to the point at  infinity on $E_T$.
 It is easy to see that  $Q_0=(T^2, T(T^2+1))$ is a point on $E_T$ corresponding to the point at infinity on $C_T$.
 Moreover, one can easily check that the point
$$Q_1= \displaystyle \left(\frac{T^2}{4}, \frac{T(T^2-8)}{8}\right) $$
is another point on $C_T$ corresponding  to the  non-torsion point
$$P_1= \displaystyle \left(\frac{T^2+4}{3 T^2}, \frac{2(T^2+1)(T^2-8)}{9 T^3}\right). $$
The point $P_1$ gives us the following triangle
\begin{equation*}
\label{eqt7}
(c,e,f)_1= \displaystyle \left( \frac{T(T^2 +16)}{ 2 (T^2-8)} ,
 \frac{T^6 + 96 T^4 + 256}{ 6 T (T^2-8)(T^2+ 4)},\frac{2 (T^2+ 1)(T^2-8)}{ 3 T (T^2+ 4)}\right).
\end{equation*}
In order to have a rational triangle with positive sides satisfying the condition $c< e <f$, we may change their rules of $c, e,$ and $f$ or replace  $-P_1$ with $P_1$ if necessary.
Then, for any integer $\ell \geq 1$, we let $P_\ell=[\ell]P_1$ be the $\ell $-th
 multiple of $P_1$ and denote its
corresponding rational triangle by  $(c,e,f)_\ell$.
Letting  $ T = 2 n \sqrt{r^2-s^2}$  and  using  Equations~\eqref{eqt-1}
 one can obtain infinity many  distinct  $\ta$-parallelogram  envelopes
for each $n\geq 1$ and for any Pythagorean angle $\ta \in (0,\pi/2]$ with $\cos(\ta)=s/r$.
Therefore,  the part~(ii) is proved.

The last part is a direct consequence of the part~(ii) if we define  $m=b/d$  for the infinitely many $\ta$-parallelogram  envelopes  corresponding to
the points $[\ell]P_1$ for $\ell\geq1$.
\begin{rem}
For $\ta=\pi/2$ and another angle  $\ta$ with $\cos(\ta)=3/5$, using the proof of Theorem~\ref{main9},  we provide some  $\ta$-parallelogram envelopes for square-free natural numbers $1\leq n\leq 50$ in  Tables~\ref{tab-1} and \ref{tab-2}, respectively.
In those tables, the bold numbers are neither $\ta$- and nor $\pi-\ta$-congruent numbers.
\end{rem}
\begin{exam}
	\label{exam1}
	For $\ta=\pi/3$, all the  square-free natural numbers $1\leq n \leq 50$ which are neither $\pi/3$- nor $2\pi/3$-congruent numbers
	but belong to  $\Nc_{\pi/3}$ are $n=2, 3, 7, 26, 31,$ and $38.$
	See Table~\ref{tab:example} in the Appendix for corresponding $\pi/3$-parallelogram envelopes for all square-free natural numbers $1\leq n \leq 50$, which
	are found by an ad hoc searching for the  components.
\end{exam}

\section*{Acknowledgments}
The first named author would likes to thank for the hospitality of Institute of Advanced Studies in Basic Sciences (IASBS) during his sabbatical year as a postdoctoral researcher  supported by the Iranian National Elites Foundation.
The second named author also thanks for the hospitality and partial financial support of IASBS.

\section*{Appendix}
	\begin{table}[h]
\caption{More data on the elliptic curve $\mathcal G_\ta^m$}
\label{tab:data}  
		\begin{center}
			\begin{tabular}{|l|c|l|c|l|c|l|}
				\hline
				& \multicolumn{2}{c|}{$(r,s)=(2,1)$} & \multicolumn{2}{c|}{$(r,s)=(3,1)$} &
				\multicolumn{2}{c|}{$(r,s)=(4,1)$}\\ \hline
				$m$   & $r_\ta(m)$ & $\Gg_{\ta^m}(\Q)_{\rm{tors}}$   & $r_\ta(m)$ & $\Gg_{\ta^m}(\Q)_{\rm{tors}}$  & $r_\ta(m)$ & $\Gg_{\ta^m}(\Q)_{\rm{tors}}$ \\ \hline
				${1}/{2}$ & 1 & $\Z/4\Z$ & 0 & $\Z/8\Z$ & 1 & $\Z/4\Z$ \\
				${1}/{3}$ & 0 & $\Z/8\Z$ & 1 & $\Z/4\Z$ & 1 & $\Z/4\Z$ \\
				${1}/{4}$ & 1 & $\Z/4\Z$ & 1 & $\Z/4\Z$ & 1 & $\Z/4\Z$ \\
				${1}/{5}$ & 1 & $\Z/4\Z$ & 1 & $\Z/4\Z$ & 2 & $\Z/4\Z$ \\
				${1}/{6}$ & 2 & $\Z/4\Z$ & 1 & $\Z/4\Z$ & 1 & $\Z/4\Z$ \\
				${1}/{7}$ & 1 & $\Z/4\Z$ & 2 & $\Z/4\Z$ & 1 & $\Z/4\Z$ \\
				${1}/{8}$ & 1 & $\Z/4\Z$ & 1 & $\Z/4\Z$ & 1 & $\Z/4\Z$ \\
				${1}/{9}$ & 1 & $\Z/4\Z$ & 1 & $\Z/4\Z$ & 2 & $\Z/4\Z$ \\
				${1}/{10}$ & 1 & $\Z/4\Z$ & 1 & $\Z/4\Z$ & 1 & $\Z/4\Z$\\
				${2}/{3}$ & 1 & $\Z/4\Z$ & 1 & $\Z/4\Z$ & 1 & $\Z/4\Z$ \\
				${2}/{5}$ & 1 & $\Z/4\Z$ & 1 & $\Z/2\Z\times\Z/4\Z$ & 1 & $\Z/4\Z$ \\
				${2}/{7}$ & 1 & $\Z/4\Z$ & 1 & $\Z/4\Z$ & 1 & $\Z/4\Z$ \\
				${2}/{9}$ & 1 & $\Z/4\Z$ & 1 & $\Z/4\Z$ & 1 & $\Z/4\Z$ \\
				${3}/{4}$ & 1 & $\Z/4\Z$ & 1 & $\Z/4\Z$ & 1 & $\Z/2\Z\times\Z/4\Z$ \\
				${3}/{5}$ & 1 & $\Z/4\Z$ & 1 & $\Z/4\Z$ & 0 & $\Z/8\Z$ \\
				${3}/{7}$ & 2 & $\Z/4\Z$ & 2 & $\Z/4\Z$ & 1 & $\Z/4\Z$ \\
				${3}/{8}$ & 1 & $\Z/2\Z\times\Z/4\Z$ & 2 & $\Z/4\Z$ & 2 & $\Z/4\Z$ \\
				${3}/{10}$ & 1 & $\Z/4\Z$ & 1 & $\Z/4\Z$ & 1 & $\Z/4\Z$  \\
				${4}/{5}$ & 1 & $\Z/4\Z$ & 1 & $\Z/4\Z$ & 1 & $\Z/4\Z$  \\
				${4}/{7}$ & 1 & $\Z/4\Z$ & 1 & $\Z/4\Z$ & 1 & $\Z/2\Z\times\Z/4\Z$ \\
				${4}/{9}$ & 2 & $\Z/4\Z$ & 1 & $\Z/4\Z$ & 1 & $\Z/4\Z$ \\
				${5}/{6}$ & 2 & $\Z/4\Z$ & 1 & $\Z/4\Z$ & 1 & $\Z/4\Z$ \\
				${5}/{7}$ & 1 & $\Z/4\Z$ & 1 & $\Z/4\Z$ & 1 & $\Z/4\Z$ \\
				${5}/{8}$ & 1 & $\Z/2\Z\times\Z/4\Z$ & 2 & $\Z/4\Z$ & 2 & $\Z/4\Z$  \\
				${5}/{9}$ & 1 & $\Z/4\Z$ & 1 & $\Z/2\Z\times\Z/4\Z$ & 1 & $\Z/4\Z$ \\
				${6}/{7}$ & 1 & $\Z/4\Z$ & 2 & $\Z/4\Z$ & 2 & $\Z/4\Z$ \\
				${7}/{8}$ & 2 & $\Z/4\Z$ & 1 & $\Z/4\Z$ & 2 & $\Z/4\Z$ \\
				${7}/{9}$ & 1 & $\Z/4\Z$ & 1 & $\Z/4\Z$ & 1 & $\Z/4\Z$ \\
				${7}/{10}$ & 1 & $\Z/4\Z$ & 1 & $\Z/2\Z\times\Z/4\Z$ & 1 & $\Z/4\Z$ \\
				${8}/{9}$ & 2 & $\Z/4\Z$ & 1 & $\Z/4\Z$ & 1 & $\Z/4\Z$ \\
				${9}/{10}$ & 2 & $\Z/4\Z$ & 1 & $\Z/4\Z$ & 2 & $\Z/4\Z$ \\ \hline
			\end{tabular}
		\end{center}
	\end{table}
\begin{table}[htbp]
\caption{Some $\pi/2$-parallelogram envelopes for	 square-free $1\leq n\leq 50$}
\label{tab-1}
	\begin{tabular}{ll}
		\toprule
		$n$ & $(a,b,c,d,e)_{\frac{\pi}{2}}$  \\
		\midrule
	  ${\bf 1}$ & $\left(\frac{2}{5},{\frac{7}{60}},{\frac{5}{12}},{\frac{143}{60}},{\frac{29}{12}}
		 \right)$\\ \midrule
		${{\bf 2}}$ & $\left(\frac{1}{2},{\frac{4}{15}},{\frac{17}{30}},{\frac{56}{15}},{\frac{113}{30}} \right)$ \\ \midrule
		${{\bf 3}}$ & $\left({\frac{14}{13}},{\frac{2233}{2340}},{\frac{259}{180}},{\frac{29999}{
				16380}},{\frac{2677}{1260}}\right)$\\ \midrule
		$5$ & $\left({\frac{46}{29}},{\frac{57017}{22620}},{\frac{2323}{780}},{\frac{
					328559}{520260}},{\frac{30629}{17940}}\right)$\\ \midrule
		$6$ & $\left({\frac{17}{10}},{\frac{5474}{1665}},{\frac{2465}{666}},{\frac{6842}{
				28305}},{\frac{19441}{11322}}	\right)$ \\ \midrule
		$7$ & $\left({\frac{14700}{9259}},{\frac{15470431}{19443900}},{\frac{175277}{98700
		}},{\frac{66913}{18518}},{\frac{371}{94}}\right)$ \\ \midrule
		${\bf 10}$ & $\left({\frac{30300}{19649}},{\frac{83810401}{59536470}},{\frac{310001}{
				148470}},{\frac{99760}{19649}},{\frac{260}{49}}	\right)$ \\ \midrule
	${\bf 11}$ & $\left({\frac{88572}{57715}},{\frac{742637383}{464721180}},{\frac{2122949}{
			958188}},{\frac{642917}{115430}},{\frac{1375}{238}}\right)$ \\ \midrule
		$13$ & $\left({\frac{172380}{113059}},{\frac{2949868471}{1499162340}},{\frac{
				5512277}{2214420}},{\frac{1483027}{226118}},{\frac{2249}{334}}\right)$ \\ \midrule
		$14$ & $\left({\frac{115836}{76145}},{\frac{1353665377}{630023730}},{\frac{2112881}
			{802578}},{\frac{537152}{76145}},{\frac{700}{97}}\right)$ \\ \midrule
		$15$ & $\left({\frac{305100}{200923}},{\frac{9512390479}{4086773820}},{\frac{
				12605629}{4535820}},{\frac{3034185}{401846}},{\frac{3435}{446}}	\right)$ \\ \midrule
		${\bf 17}$ & $\left({\frac{502860}{332059}},{\frac{26336096911}{9822305220}},{\frac{
				26142077}{8489460}},{\frac{5674583}{664118}},{\frac{4981}{574}}	\right)$ \\ \midrule
		${\bf 19}$ & $\left({\frac{784092}{518755}},{\frac{64879403383}{21407981340}},{\frac{
				50173589}{14815212}},{\frac{9897613}{1037510}},{\frac{6935}{718}}\right)$ \\ \midrule
		$21$ & $\left({\frac{1169532}{774835}},{\frac{145721314183}{43152110820}},{\frac{
				90433669}{24448788}},{\frac{16327227}{1549670}},{\frac{9345}{878}}	\right)$ \\ \midrule
		$22$ & $\left({\frac{704220}{466817}},{\frac{53034857569}{14942812170}},{\frac{
					29750513}{7714410}},{\frac{5150992}{466817}},{\frac{2684}{241}}	\right)$ \\ \midrule
		$23$ & $\left({\frac{1682220}{1115659}},{\frac{303628332991}{81599299260}},{\frac{
				154752077}{38544780}},{\frac{25733297}{2231318}},{\frac{12259}{1054}}\right)$ \\ \midrule
		${\bf 26}$ & $\left({\frac{1372956}{911585}},{\frac{203809927777}{48137157510}},{\frac{
				79970801}{17795622}},{\frac{11877008}{911585}},{\frac{4420}{337}}\right)$ \\ \midrule
		$29$ & $\left({\frac{4248732}{2823235}},{\frac{1962339124183}{413626513380}},{\frac
			{611798069}{122920212}},{\frac{82020323}{5646470}},{\frac{24505}{1678}}	\right)$ \\ \midrule
		$30$ & $\left({\frac{2432700}{1616849}},{\frac{644258613601}{131110285410}},{\frac{
				187110001}{36409410}},{\frac{24293280}{1616849}},{\frac{6780}{449}}
		\right)$ \\ \midrule
		$31$ & $\left({\frac{5546892}{3687355}},{\frac{3353898590383}{659785804860}},{\frac
			{909668189}{171595788}},{\frac{114486937}{7374710}},{\frac{29915}{1918}}
		\right)$ \\ \midrule
		${\bf 33}$ & $\left({\frac{7122060}{4736059}},{\frac{5541717534511}{1022136253380}},{
			\frac{1319930077}{234596340}},{\frac{156505767}{9472118}},{\frac{36069
			}{2174}} \right)$ \\ \midrule
		$34$ & $\left({\frac{4012476}{2668625}},{\frac{1760698855777}{314935110750}},{\frac
			{394219121}{68094078}},{\frac{45425632}{2668625}},{\frac{9860}{577}}
		\right)$ \\ \midrule
		${\bf 35}$ & $\left({\frac{9011100}{5993923}},{\frac{8888050331479}{1543195415580}},{
			\frac{1874280629}{314873580}},{\frac{210044765}{11987846}},{\frac{
				43015}{2446}}\right)$ \\ \midrule
		$37$ & $\left({\frac{11253180}{7487059}},{\frac{13883190600871}{2277114124260}},{
			\frac{2610706277}{415759380}},{\frac{277325323}{14974118}},{\frac{
				50801}{2734}}\right)$ \\ \midrule
		$38$ & $\left({\frac{6259740}{4165217}},{\frac{4298876938849}{686136196410}},{\frac
			{765244913}{118770330}},{\frac{79221488}{4165217}},{\frac{13756}{721}}
		\right)$ \\ \midrule
		$39$ & $\left( {\frac{13889772}{9243115}},{\frac{21179374244383}{3291916921020}},{
			\frac{3574266349}{540988812}},{\frac{360837633}{18486230}},{\frac{
				59475}{3038}}
		 \right)$ \\ \midrule
		$41$ & $\left( {\frac{16964652}{11291275}},{\frac{31631011985983}{4672013439300}},{
			\frac{4817922509}{694723188}},{\frac{463356047}{22582550}},{\frac{
				69085}{3358}}\right)$ \\ \midrule
	${\bf 42}$ & $\left({\frac{9340380}{6217217}},{\frac{9593473907809}{1382646888630}},{
		\frac{1390928113}{195925590}},{\frac{130672752}{6217217}},{\frac{18564
		}{881}}\right)$ \\ \midrule
		${\bf 43}$ & $\left({\frac{20523900}{13662259}},{\frac{46341994593031}{6520996220700}},{
			\frac{6403414277}{881573100}},{\frac{587954437}{27324518}},{\frac{
				79679}{3694}}\right)$ \\ \midrule
		$46$ & $\left({\frac{13438716}{8947505}},{\frac{19893681780577}{2613977795730}},{
			\frac{2395438961}{308798322}},{\frac{205938688}{8947505}},{\frac{24380}{1057}}	\right)$ \\ \midrule
		$47$ & $\left({\frac{29291340}{19503259}},{\frac{94544614798351}{12154821073980}},{
			\frac{10896327677}{1375446540}},{\frac{917276393}{39006518}},{\frac{
				104011}{4414}}\right)$ \\
		\bottomrule
	\end{tabular}
\end{table}
\begin{table}[htbp]
\caption{Some $\ta$-parallelogram envelopes for square-free $1\leq n\leq 50$, where $\cos(\ta)=3/5$}
\label{tab-2}
	\begin{tabular}{lccl}
		\toprule
		$n$ & $[r_1,t_1]$ & $[r_2,t_2]$ & $(a,b,c,d,e)_{\ta}$  with $\cos(\ta)=3/5$ \\
		\midrule
		$1$ & $[0,4]$ & $[1,4]$ & $\left({\frac{204}{91}},{\frac{25085}{18564}},{\frac{2561}{1428}},{\frac{80}	{91}},{\frac{20}{7}}\right)$\\ \midrule
		${\bf 2}$ & $[0,4]$ & $[0,4]$ & $\left({\frac{15600}{7967}},{\frac{27143809}{12428520}},{\frac{90113}{48360}},{\frac{23288}{7967}},{\frac{136}{31}}\right)$\\ \midrule
		$3$ & $[1,4]$ & $[0,4]$ & $\left({\frac{78300}{40967}},{\frac{626820049}{213847740}},{\frac{870913}{	 370620}},{\frac{201432}{40967}},{\frac{444}{71}}\right)$\\ \midrule
		$5$ & $[0,4]$ & $[1,4]$ & $\left({\frac{601500}{318599}},{\frac{33244215601}{7665491940}},{\frac{16960001}{4787940}},{\frac{2837120}{318599}},{\frac{2020}{199}}\right)$\\ \midrule
		$6$ & $[1,4]$ & $[0,4]$ &  $\left({\frac{249264}{132307}},{\frac{3943947275}{785223144}},{\frac{49766401}{11923128}},{\frac{206040}{18901}},{\frac{3480}{287}}\right)$ \\ \midrule
		$7$ & $[1,4]$ & $[0,4]$ & $\left({\frac{2307900}{1226567}},{\frac{461284727569}{80879827980}},{\frac{124160513}{25782540}},{\frac{15820168}{1226567}},{\frac{5516}{391}}\right)$ \\ \midrule
		$10$ & $[1,4]$ & $[0,4]$ & $\left({\frac{9606000}{5114399}},{\frac{7594358198401}{982578335880}},{\frac	 {1039360001}{153503880}},{\frac{96620440}{5114399}},{\frac{16040}{799}}\right)$ \\ \midrule
		$11$ & $[0,4]$ & $[1,4]$ & $\left({\frac{2812524}{1497883}},{\frac{3217536636125}{382984716972}},{\frac	 {1836567041}{247246428}},{\frac{31291480}{1497883}},{\frac{21340}{967}}	\right)$ \\ \midrule
		$13$ & $[1,4]$ & $[0,4]$ & $\left({\frac{27428700}{14613767}},{\frac{8583302528647}{880959628380}},{\frac{4986522113}{570094980}},{\frac{51958816}{2087681}},{\frac{35204}
			{1351}}\right)$ \\ \midrule
		${\bf 14}$ & $[0,4]$ & $[0,4]$ & $\left({\frac{7378224}{3931603}},{\frac{107881530976129}{10360088433240}},{\frac{7769251841}{825834072}},{\frac{528553256}{19658015}},{\frac{
				43960}{1567}}\right)$ \\ \midrule
		$15$ & $[0,4]$ & $[1,4]$ & $\left({\frac{48613500}{25907399}},{\frac{186109696832401}{16792657883820}},	 {\frac{11741760001}{1166075820}},{\frac{748377960}{25907399}},{\frac{
				54060}{1799}}\right)$ \\ \midrule
		$17$ & $[1,4]$ & $[0,4]$ & $\left({\frac{80197500}{42746567}},{\frac{500956576117489}{40331385964500}},	 {\frac{24845158913}{2180428500}},{\frac{1405736528}{42746567}},{\frac{
				78676}{2311}}\right)$ \\ \midrule
		$19$ & $[0,4]$ & $[1,4]$ & $\left({\frac{25025964}{13340827}},{\frac{1208792909726929}{87859751640060}}, {\frac{48375155201}{3802629372}},{\frac{2460337816}{66704135}},{\frac
			{109820}{2887}}\right)$ \\ \midrule
		$21$ & $[2,4]$ & $[1,4]$ & $\left({\frac{37345644}{19909915}},{\frac{106891722534361}{7081415215812}}, {	 \frac{88123230721}{6272289828}},{\frac{162797376}{3981983}},{\frac{
				148260}{3527}}\right)$ \\ \midrule
		$22$ & $[0,4]$ & $[1,4]$ &  $\left({\frac{224914800}{119911967}},{\frac{3864726317451649}{
				245181600685560}},{\frac{116460838913}{7914956280}},{\frac{5142200008}
			{119911967}}, {\frac{170456}{3871}}\right)$ \\ \midrule
		$23$ & $[1,4]$ & $[0,4]$ & $\left({\frac{268679100}{143248967}},{\frac{5498971329569809}{
				334678291560780}},{\frac{152018586113}{9885054540}},{\frac{6429413192}
			{143248967}}, {\frac{194764}{4231}}\right)$ \\ \midrule
		$26$ & $[1,4]$ & $[0,4]$ & $\left({\frac{87747504}{46786771}},{\frac{2910709595735885}{157900860594984}}, {\frac{317031669761}{18248105928}},{\frac{2380601080}{46786771}},{
			\frac{281320}{5407}}\right)$ \\ \midrule
		$29$ & $[0,4]$ & $[1,4]$ & $\left({\frac{135808044}{72416155}},{\frac{34652540443428529}{
				1695637304232900}}, {\frac{610185464321}{31502783172}},{\frac{
				20595602336}{362080775}},{\frac{390340}{6727}}\right)$ \\ \midrule
		${\bf 30}$ & $[0,4]$ & $[0,4]$ & $\left({\frac{777654000}{414669599}},{\frac{45368631579081601}{
				2149796482271640}},{\frac{747740160001}{37322207640}},{\frac{
				24416175720}{414669599}},{\frac{432120}{7199}}\right)$ \\ \midrule
		$31$ & $[2,4]$ & $[1,4]$ & $\left({\frac{177327564}{94557787}},{\frac{11775813401066525}{
				540893613740028}},{\frac{910222297601}{43971515628}},
	{\frac{5756758280	}{94557787}},{\frac{476780}{7687}}\right)$ \\ \midrule
		$33$ & $[1,4]$ & $[0,4]$ & $\left({\frac{1138549500}{607130567}},
		{\frac{96798236787579889}{4189383051470100}},{\frac{1324284774913}{60108513300}},{\frac{39390937872}{607130567}},{\frac{575124}{8711}}\right)$ \\ \midrule
		$34$ & $[1,4]$ & $[1,4]$ & $\left({\frac{256590384}{136827859}},{\frac{17534322811061047}{
				737578001737560}},{\frac{1583932334081}{69785037672}},{\frac{
				6536496328}{97734185}},{\frac{629000}{9247}}\right)$ \\ \midrule
		$35$ & $[0,4]$ & $[1,4]$ & $\left({\frac{1440673500}{768251399}},{\frac{154572690545960401}{6324568182155580}},{\frac{1884688960001}{80669483580}},{\frac{
				52917310040}{768251399}},{\frac{686140}{9799}}\right)$ \\ \midrule
		$37$ & $[1,4]$ & $[2,4]$ & $\left({\frac{1799276700}{959493767}},{\frac{240522746409818929}{9331863668963940}},{\frac{2630182554113}{106507454820}},{\frac{
				69927835648}{959493767}},{\frac{810596}{10951}}\right)$ \\ \midrule
		$38$ & $[0,4]$ & $[1,4]$ & $\left({\frac{2001817200}{1067508767}},{\frac{297391044346297729}{11247144268059960}},{\frac{3086401626113}{121699949880}},{\frac{
				79934843912}{1067508767}},{\frac{878104}{11551}}\right)$ \\ \midrule
		$39$ & $[0,4]$ & $[1,4]$ & $\left({\frac{444198924}{236879323}},{\frac{365692973639290129}{13489941076211340}},{\frac{3606747056641}{138578674572}},{\frac{91056034056}{1184396615}},{\frac{949260}{12167}}\right)$ \\ \midrule
		$41$ & $[2,4]$ & $[1,4]$ & $\left({\frac{542566284}{289339099}},{\frac{15557821814308955}{
				546988291847868}},{\frac{4868447111681}{177948507828}},{\frac{
				3343084240}{41334157}},{\frac{1102900}{13447}}\right)$ \\ \midrule
		$42$ & $[0,4]$ & $[1,4]$ & $\left({\frac{2987334000}{1593089567}},{\frac{659694548714076289}{22662336326401800}},{\frac{5625548070913}{200734619400}},{\frac{
				132034235928}{1593089567}},{\frac{1185576}{14111}}\right)$ \\ \midrule
		$43$ & $[1,4]$ & $[0,4]$ & $\left({\frac{3282159900}{1750322567}},{\frac{795640144838322769}{26720179262662620}},{\frac{6478327040513}{225797335260}},{\frac{
				148566077752}{1750322567}},{\frac{1272284}{14791}}\right)$ \\ \midrule
		${\bf 46}$ & $[0,4]$ & $[0,4]$ & $\left({\frac{859696944}{458467795}},{\frac{54456420261173161}{
				1713666792538776}},{\frac{9708557393921}{316349786328}},{\frac{
				8333017576}{91693559}},{\frac{1557560}{16927}}\right)$ \\ \midrule
		$47$ & $[1,4]$ & $[0,4]$ & $\left({\frac{4684626300}{2498272967}},{\frac{1615810660236674449}{49802022322498860}},{\frac{11045411686913}{352263963180}},{\frac{232036850408}{2498272967}},{\frac{1661356}{17671}}\right)$ \\
		\bottomrule
	\end{tabular}
\end{table}
\begin{table}[htbp]
 \caption{Some $\pi/3$-parallelogram envelopes for square-free	$1\leq n\leq 50$}
	\label{tab:example}
	\begin{tabular}{lccl}
			\toprule
		$n$ & $[r_1,t_1]$ & $[r_2,t_2]$ & $(a,b,c,d,e)_{\pi/3}$  \\
			\midrule
		$1$ &  $[0,8]$ & $[0,4]$ & $\left(\frac{5}{11}, \frac{3}{8}, \frac{37}{88}, \frac{161}{40}, \frac{1879}{440}\right)$\\ \midrule
		$\bf{2}$ &  $[0,4]$ & $[0,4]$ & $\left(\frac{360}{209}, \frac{567}{418}, \frac{657}{418}, \frac{9083}{9405}, \frac{22183}{9405}\right)$ \\ \midrule
		$\bf{3}$ &  $[0,4]$ & $[0,4]$ & $\left(2, \frac{5}{4}, \frac{7}{4}, \frac{7}{4}, \frac{13}{4}\right)$\\ \midrule
		$5$ & $[0,4]$ & $[1,4]$ & $\left(\frac{24}{7}, \frac{5}{2}, \frac{43}{14}, \frac{5}{12}, \frac{307}{84}\right)$ \\ \midrule
		$6$ & $[1,4]$ & $[0,4]$ & $\left(\frac{8}{5}, \frac{7}{6}, \frac{43}{30}, \frac{19}{3}, \frac{109}{15}\right)$ \\ \midrule
		${\bf{7}}$ & $[0,4]$ & $[0,4]$ & $\left(\frac{143}{42}, \frac{20}{7}, \frac{19}{6}, \frac{1256}{1001}, \frac{3583}{858}\right)$ \\ \midrule
		$10$ &  $[1,4]$ & $[1,4]$ & $\left(\frac{40}{7}, \frac{15}{7}, 5, \frac{19}{14}, \frac{13}{2}\right)$ \\ \midrule
		$11$ &  $[1,4]$ & $[0,4]$ & $\left(\frac{44}{15}, \frac{11}{10}, \frac{77}{30}, \frac{32}{5}, \frac{124}{15}\right)$ \\ \midrule
		$13$ & $[1,4]$ & $[0,4]$ & $\left(\frac{13}{3}, \frac{13}{8}, \frac{91}{24}, \frac{35}{8}, \frac{181}{24}\right)$\\ \midrule
		$14$ & $[0,4]$ & $[2,4]$ & $\left(\frac{14}{5}, \frac{7}{4}, \frac{49}{20}, \frac{33}{4}, \frac{199}{20}\right)$ \\ \midrule
		$15$ & $[0,4]$ & $[1,4]$ & $\left(6, \frac{9}{4}, \frac{21}{4}, \frac{11}{4}, \frac{31}{4}\right)$
		\\ \midrule
		$17$ & $[1,4]$ & $[1,4]$ & $\left(\frac{15}{4}, \frac{52}{15}, \frac{217}{60}, \frac{28}{5}, \frac{163}{20}\right)$ \\ \midrule
		$19$ &  $[0,4]$ & $[1,4]$ & $\left(\frac{57}{40}, \frac{9}{7}, \frac{381}{280}, \frac{533}{21}, \frac{21943}{840}\right)$ \\ \midrule
		$21$ & $[1,4]$ & $[1,4]$ & $\left(\frac{21}{4}, \frac{14}{5}, \frac{91}{20}, \frac{26}{5}, \frac{181}{20}\right)$ \\ \midrule
		$22$ & $[1,4]$ & $[1,4]$ & $\left(8, 3, 7, \frac{5}{2}, \frac{19}{2}\right)$
		\\ \midrule
		$23$ & $[1,4]$ & $[1,4]$ &  $\left(\frac{92}{35}, \frac{529}{220}, \frac{3887}{1540}, \frac{3321}{220}, \frac{25513}{1540}\right)$ \\ \midrule
		$\bf{26}$ & $[0,4]$ & $[0,4]$ & $\left(\frac{104}{15}, \frac{104}{15}, \frac{104}{15}, \frac{17}{30}, \frac{217}{30}\right)$ \\ \midrule
		$29$ & $[0,4]$ & $[1,4]$ & $\left(\frac{15}{4}, \frac{52}{15}, \frac{217}{60}, 12, \frac{57}{4}\right)$ \\ \midrule
		$30$ & $[1,4]$ & $[0,4]$ & $\left(\frac{60}{7}, \frac{75}{14}, \frac{15}{2}, \frac{23}{14}, \frac{19}{2}\right)$ \\ \midrule
		$\bf{31}$ & $[0,4]$ & $[0,4]$ & $\left(\frac{1860}{343}, \frac{10013}{2058}, \frac{31}{6}, \frac{11264}{1715}, \frac{52}{5}\right)$ \\ \midrule
		$33$ & $[0,4]$ & $[1,4]$ & $\left(\frac{60}{7}, \frac{128}{17}, \frac{964}{119}, \frac{29}{170}, \frac{10303}{1190}\right)$ \\ \midrule
		$34$ &  $[1,4]$ & $[1,4]$ & $\left(\frac{34}{15}, \frac{34}{15}, \frac{34}{15}, \frac{416}{15}, \frac{434}{15}\right)$ \\ \midrule
		$35$ & $[1,4]$ & $[0,4]$ & $\left(\frac{140}{13}, \frac{224}{39}, \frac{28}{3}, \frac{59}{78}, \frac{67}{6}\right)$ \\ \midrule
		$37$ & $[1,4]$ & $[0,4]$ & $\left(\frac{222}{35}, \frac{333}{140}, \frac{111}{20}, \frac{3901}{420}, \frac{817}{60}\right)$ \\ \midrule
		$\bf{38}$ & $[0,4]$ & $[0,4]$ &  $\left(\frac{3}{2}, \frac{4}{5}, \frac{13}{10}, \frac{748}{15}, \frac{1519}{30}\right)$ \\ \midrule
		$39$ & $[2,4]$ & $[1,4]$ & $\left(\frac{65}{6}, \frac{325}{48}, \frac{455}{48}, \frac{103}{240}, \frac{2653}{240}\right)$ \\ \midrule
		$41$ &$[1,4]$ & $[1,4]$ & $\left(\frac{60}{7}, \frac{832}{105}, \frac{124}{15}, \frac{23}{14}, \frac{19}{2}\right)$ \\ \midrule
		$42$ & $[1,4]$ & $[0,4]$ & $\left(\frac{28}{5}, \frac{7}{2}, \frac{49}{10}, \frac{23}{2}, \frac{151}{10}\right)$\\ \midrule
		$43$ & $[0,4]$ & $[1,4]$ & $\left(\frac{14}{3}, \frac{77}{20}, \frac{259}{60}, \frac{2041}{140}, \frac{7303}{420}\right)$ \\ \midrule
		$46$ & $[1,4]$ & $[0,4]$ & $\left(\frac{15}{2}, \frac{104}{15}, \frac{217}{30}, \frac{16}{3}, \frac{67}{6}\right)$ \\ \midrule
		$47$ & $[1,4]$ & $[1,4]$ & $\left(\frac{423}{56}, \frac{423}{56}, \frac{423}{56}, \frac{2465}{504}, \frac{5473}{504}\right)$ \\
		\bottomrule
	\end{tabular}
	\end{table}

In Tables~\ref{tab-1}--\ref{tab:example}, the highlighted $n$'s in boldface address they are neither $\ta$- nor $(\pi-\ta)$-congruent numbers. And in Tables~\ref{tab-2} and \ref{tab:example}, the values $r_1$ and $r_2$ denote, resp., the rank of the (corresponding) $\ta$- and $(\pi-\ta)$-congruent elliptic curves over $\Q$, while $t_1$ and $t_2$, resp., refer to the number of torsion points on the $\ta$- and $(\pi-\ta)$-congruent elliptic curves.

\end{document}